\PassOptionsToPackage{numbers,sort&compress}{natbib}
\documentclass[preprint,12pt]{elsarticle}

\usepackage[english]{babel}
\usepackage[T1]{fontenc}
\usepackage[utf8]{inputenc}

\usepackage[margin=1.1in]{geometry}
\usepackage[pages=all,color=black,position={current page.south},
            placement=bottom,scale=1,opacity=1,vshift=5mm]{background}
\SetBgContents{}

\usepackage{amsmath,amssymb,amsthm,mathtools,bm,stmaryrd,bbm,mathrsfs}
\DeclareMathOperator{\Span}{span}

\usepackage{algorithm}
\usepackage{algpseudocode}

\usepackage{tabularx,array,multirow,booktabs,diagbox}
\usepackage{float}
\usepackage[section]{placeins}
\usepackage{multicol}
\usepackage{caption}
\usepackage{subcaption}

\usepackage{graphicx}
\graphicspath{{fig/}}
\usepackage[export]{adjustbox}

\usepackage{pgfplots}
\usepackage{pgfplotstable}
\usepackage{pgf-umlcd}
\usepackage{forest}
\usetikzlibrary{er,positioning,shapes,external}
\usepgfplotslibrary{colorbrewer,groupplots,external}

\pgfplotsset{compat=1.18}
\pgfplotsset{
    cycle list/Dark2-7,
    cycle multiindex* list={
        mark list*\nextlist
        Dark2-7\nextlist
    },
    every axis plot/.append style={thick}
}

\newif\ifarxiv
\arxivfalse  %

\ifarxiv
    \tikzset{external/export=false}
\else
    \tikzset{external/only named=true}
\fi

\usepackage{xcolor}
\usepackage[normalem]{ulem}
\usepackage{comment}
\usepackage{todonotes}

\usepackage{pifont}

\usepackage{siunitx}
\theoremstyle{definition}
\newtheorem{definition}{Definition}[section]
\newtheorem{theorem}{Theorem}
\newtheorem{remark}{Remark}
\newtheorem{assumptions}{Assumptions}
\newtheorem{lemma}{Lemma}

\newcommand{\norm}[1]{\left\lVert#1\right\rVert}

\tikzset{
  bplus/.style={
    rectangle split,
    rectangle split horizontal,
    text width=1em,
    text centered,
    inner xsep=2pt,
    draw
  }
}

\forestset{
  multi edge/.style={
    edge path={
      \noexpand\path[\forestoption{edge}]
      (!u.#1 south)--(.child anchor)\forestoption{edge label};
    }
  }
}

\ExplSyntaxOn
\seq_new:N \l_tree_node_seq
\seq_set_from_clist:Nn \c_node_names_seq {one,two,three,four}
\NewDocumentCommand{\mpn}{m}{
  \seq_set_from_clist:Nn \l_tree_node_seq {#1}
  \seq_map_indexed_inline:Nn \l_tree_node_seq {
    \nodepart{\seq_item:Nn \c_node_names_seq {##1}} {##2}
  }
}
\ExplSyntaxOff

\usepackage[colorlinks=true]{hyperref}
\hypersetup{
    linkcolor=blue,
    filecolor=magenta,
    urlcolor=cyan
}

\begin{document}
\captionsetup[table]{skip=6pt}

\begin{frontmatter}
    \title{R3MG-C: a high-order algebraic-geometric multilevel preconditioner for continuous finite element discretizations}

    \author[inst1]{Davide Polverino\corref{cor1}}\ead{davide.polverino@phd.unipi.it}
    \author[inst1,inst2]{Marco Feder}\ead{marco.feder@dm.unipi.it}
    \author[inst1]{Luca Heltai}\ead{luca.heltai@unipi.it}

    \cortext[cor1]{Corresponding author.}

    \affiliation[inst1]{organization={University of Pisa},
        addressline={Largo B. Pontecorvo, 5},
        city={Pisa},
        postcode={56126},
        country={Italy}}

        \affiliation[inst2]{organization={University of Roma Tor Vergata},
        addressline={Via Politecnico, 1},
        city={Roma},
        postcode={00133},
        country={Italy}}   

    \begin{abstract}
Algebraic multigrid (AMG) methods are robust and efficient black-box preconditioners for linear systems arising from low-order discretizations of elliptic partial differential equations, but their performance often deteriorates for high-order methods. We introduce an algebraic-geometric multilevel preconditioner for continuous lagrangian finite element discretizations. The method automatically constructs prolongation operators and a Galerkin hierarchy using only the coordinates of finite element support points, without requiring a prescribed mesh hierarchy or domain decomposition. The support-point are recursively partitioned by an R-tree algorithm based on axis-aligned bounding boxes, producing a hierarchy of agglomerates. In contrast to classical AMG, which typically uses piecewise-constant aggregate coarse spaces, our method embeds local discontinuous polynomial spaces of degree $p'>0$, defined on the agglomerate boxes, through continuous nodal interpolation. This yields a high-order extension of the Nicolaides coarse space. A two-level analysis quantifies how the coarse polynomial degree offsets the effects of large agglomerates and high-order fine discretizations, under uniform box-regularity, interpolation-stability, and stable-decomposition assumptions. Numerical experiments in two and three dimensions show that the resulting V-cycle, used as a conjugate-gradient preconditioner, maintains stable iteration counts and remains effective in regimes where standard AMG deteriorates.
    \end{abstract}

\end{frontmatter}

\section{Introduction}\label{sec:intro}

High-order finite element discretizations of partial differential equations (PDEs) are attractive because they
deliver high accuracy per degree of freedom and are particularly
effective on smooth solutions and geometrically complex domains.
However, the efficient solution of the resulting linear systems remains
a central difficulty. Standard algebraic multigrid (AMG) methods are
designed to construct a hierarchy directly from the matrix graph and
have been extremely successful as black-box solvers for sparse systems
arising from low-order discretizations of elliptic PDEs~\cite{Brandt1986AMGTheory,RugeStuben1987AMG,Stuben2001Review,xu:review}.
For high-order finite elements, the situation is more delicate. The
matrix graph contains strong intra-element couplings and high-frequency
polynomial modes that are not represented adequately by coarse spaces
designed for low-order discretizations. As a consequence, standard AMG
can lose robustness as the polynomial degree is increased~\cite{amg:high_order,HariStadlerBirosHighOrderMG}.

The main contribution of this work is the introduction of a high-order multilevel preconditioner that employs a high-order extension of the
Nicolaides coarse-space construction~\cite{victorita::DD} to build the hierarchy. In the original Nicolaides approach~\cite{Nicolaides1987Deflation}, one augments the conjugate gradient iteration~\cite{Stiefel::CG} by coarse functions that are constant on subdomains or aggregates. This captures the lowest-order near-null space components and provides an effective mechanism for deflation and coarse correction. We generalize this idea by replacing the local
constant space by a local polynomial space of degree $p'$. The case
$p'=0$ recovers the Nicolaides-type aggregate space, whereas
$p'>0$ enriches the coarse correction with polynomial modes that are
essential for high-order finite element discretizations. This gives a
coarse space that is still algebraic-geometric and inexpensive to build,
but that retains the polynomial information needed for robustness with
respect to the fine-grid degree $p$. Enriching the coarse space is not a new idea, and many works in the domain decomposition literature employ this technique. For example, in~\cite{local_coarse_victorita1, local_coarse_victorita2, Efendiev_local_coarse, Efendiev_local_coarse2, scheichl_genEO, victorita_scheichl_local}, local generalized eigenvalue problems are solved to augment the coarse space, while in~\cite{energy_min_coarse}, energy minimization problems are solved. By contrast, our extended coarse space does not require solving any additional problem.

The construction of the coarse space is based on agglomerates generated by an R-tree hierarchy; thus, no prescribed coarse mesh is needed. R-trees and R\textsuperscript{*}-trees provide a practical way of grouping
geometric objects through bounding boxes~\cite{guttman::rtree,beckmann::rtree,manolopoulos::rtree}.
In the present setting, the objects are either cells or finite element
support points. Each agglomerate defines a local axis-aligned bounding
box, and on that box we define a tensor-product polynomial space of
degree $p'$. The coarse functions are then obtained by restricting these
box polynomials to the agglomerate and injecting them into the fine
finite element space through interpolation at the owned support points.
This leads to sparse prolongation operators and to a Galerkin hierarchy
\[
  A_{\ell+1}=P_\ell^T A_\ell P_\ell .
\]
The method therefore combines geometric information from the finite
element discretization with an algebraic multilevel construction. The use of R\textsuperscript{*}-trees to build a multilevel hierarchy was introduced in~\cite{Feder::R3MG} in the context of discontinuous Galerkin discretizations. In the conforming case considered here, however, the agglomeration procedure requires additional care.

Our algorithm is related to aggregation-based AMG and AMGe methods, where coarse spaces are built from aggregates, element interpolation, or local spectral information~\cite{VanekMandelBrezina1996SmoothedAggregation,VanekBrezinaMandel2001ConvergenceSA,
BrezinaEtAl2001AMGe,JonesVassilevski2001AMGeAgglomeration,ChartierEtAl2003SpectralAMGe,
Notay2006AggregationPreconditioning}.
The distinction is that the present coarse space is explicitly designed
as a polynomial enrichment of the Nicolaides aggregate space: rather
than being merely an algebraic interpolation of low-order unknowns (as
in standard AMG), the coarse basis is constructed as a continuous interpolation of a discontinuous high-order tensor-product approximation space
defined on Cartesian bounding boxes and tied to the support points of
the original finite element discretization. The
construction is also related in spirit to multigrid methods for spectral
and high-order elements
\cite{RonquistPatera1987SpectralMG,HariStadlerBirosHighOrderMG},
but it does not require a nested sequence of geometric meshes.

We analyze the method in a two-level subspace-correction framework~\cite{xu:review, XU_ssc, zikatanov_convergence_bound}. Let
$h$ denote the fine mesh size, $H$ the maximal diameter of an agglomerate,
$p$ the degree of the fine finite element space, and $p'\le p$ the
degree of the local coarse polynomial space. Under uniform assumptions on the aspect ratios and nondegeneracy of the bounding boxes, the stability of the local interpolation, the admissibility of the agglomeration, and a stable nodal decomposition, the two-level error propagation operator satisfies the estimate
\[
  \|E\|_A^2
  \le
  1-
  \frac{C}{
  p^{2d+8}\frac{H^3}{h^3}(1 + {p'}^{-2} + L^2p^d)(1 + L^2\frac{H}{h})c^D}
  ,
\]
where $d=2,3$ is the spatial dimension, $c^D$ is a parameter that depends on the smoother, $C$ is a positive constant independent of $H,h,p$, and $p'$, and $L$ is the maximum number of degrees of freedom held by an agglomerate. This bound identifies the mechanism behind the method. For $p'>0$,
the approximation term improves by the factor ${p'}^{-2}$, showing that
higher-order coarse polynomials compensate for larger agglomerates and
for the high-order content of the fine space, resulting in an enriched coarse space that makes algebraic multigrid
robust for high-order finite element discretizations.

Numerical results confirm this behavior and reveal that the preconditioner performs better than the current theoretical bound suggests. We consider Poisson problems discretized with conforming finite elements in two and three dimensions, on structured and unstructured meshes such as left ventricle and liver geometries. 
The proposed R-tree-based V-cycle is used as a preconditioner for conjugate gradient and is compared with a standard AMG implementation from Trilinos~\cite{Trilinos}. The experiments show that controlling the ratio $H/h$
and increasing the coarse polynomial degree $p'$ leads to stable
iteration counts for discretizations of polynomial degree $p$. In the high-order
regimes tested here, the enriched Nicolaides hierarchy remains effective
whereas a standard low-order AMG coarse space deteriorates.
The remainder of the paper is organized as follows.
Section~\ref{sec:model_problem} introduces the model problem and the
finite element notation. Section~\ref{sec:rtree} describes the R-tree
agglomeration procedure. Section~\ref{sec:preconditioner} presents the
multilevel preconditioner and the construction of the transfer
operators. Section~\ref{sec:analysis} contains the two-level convergence
analysis and the high-order Nicolaides approximation estimate.
Section~\ref{sec:numerics} reports the numerical experiments.

\section{Notation and model problem} \label{sec:model_problem}

Let $\Omega \subset \mathbb{R}^d$, with $d=2,3$, be an open, bounded, simply connected domain with Lipschitz boundary $\partial \Omega$. We denote by $H^s(\Omega)$ the Sobolev space of index $s \geq 0$ of real-valued functions on $\Omega$, endowed with the seminorm $|\cdot|_{H^s(\Omega)}$ and norm $\norm{\cdot}_{H^s(\Omega)}$. Let $L^p(\Omega)$, $p \in [1,+\infty]$, be the standard Lebesgue space of real-valued functions on $\Omega$, equipped with the norm $\norm{\cdot}_{L^p(\Omega)}$. We consider the linear elliptic problem: find $u \in H^1(\Omega)$ such that
\begin{equation}\label{eq:poisson}
\left\{
\begin{aligned}
- \Delta &u = f,  \quad \text{in} \ \Omega, \\
&u = g, \quad \text{on} \ \partial \Omega ,
\end{aligned}
\right.
\end{equation}
with $f \in L^2(\Omega)$, $g \in H^{1/2}(\partial \Omega) $. We will assume that $g$ admits a continuous representative in $H^{1/2}(\partial \Omega)$. Setting $H_D^1 := \{u \in H^1({\Omega}) : u = 0 \ \text{on} \ \partial \Omega \}$, the weak formulation of (\ref{eq:poisson}) reads: find $u \in H^1(\Omega)$, with $u = g$ on $\partial \Omega$, such that
\begin{equation}\label{eq:weakform}
    \int_{\Omega} \nabla u \cdot \nabla v  \ d\textbf{x} = \int_{\Omega} fv \ d\textbf{x}
\end{equation}
for all $v \in H_{D}^1(\Omega) $. The well-posedness of the weak problem (\ref{eq:weakform}) is guaranteed by the Lax-Milgram lemma.

\subsection{Meshes and finite element spaces}
We consider shape-regular, geometrically conforming, quasi-uniform meshes $\mathcal{T}_h = \bigcup_{i=1}^N \tau_i$, where $h_{i} := \text{diam}(\tau_i)$ is the diameter of $\tau_i$ and $h := \max_{i = 1, \dots, N} h_i$. The open elements $\tau_i \subset \Omega$ are mutually disjoint and satisfy $\bigcup_{i=1}^N \bar{\tau}_i = \bar{\Omega}$. In view of our agglomeration procedure, we denote by $K$ a polytopic element composed of fine mesh elements $\tau_i$. We introduce the conforming \emph{finite element space}
\begin{equation*}
    V_h = \Big\{ u  \in \mathcal{C}^0(\bar{\Omega}) \cap H^1(\Omega) :  \ u|_{\tau_i} \in \mathcal{R}^p(\tau_i), \tau_i \subset \Omega \ \text{and} \ u = 0 \ \text{on} \ \partial \Omega \Big\},
\end{equation*}
where $p \in \mathbb{N}$ and, if $\tau_i$ is a quadrilateral or hexahedral element, $\mathcal{R}^p(\tau_i) = \mathcal{Q}^{p}(\tau_i)$ is the polynomial space spanned by tensor-product Lagrange polynomials of degree $p$ in each variable in $\tau_i$. If $\tau_i$ is a simplicial element, $\mathcal{R}^p(\tau_i) = \mathcal{P}^{p}(\tau_i)$ is the polynomial space spanned by simplicial Lagrange polynomials of total degree $p$ in $\tau_i$. 

Without loss of generality we suppose that $g = 0$. If $g \neq 0$ a standard lifting argument can be used to reduce the problem to the homogeneous case. Then, the weak form reads: find $u \in V_h$ such that
\begin{equation}\label{eq:discweakform}
    A(u,v) = f(v) \quad \forall v \in V_h,
\end{equation}
with
\begin{equation*}
    A(u,v) = \int_{\Omega} \nabla u  \cdot \nabla v \ d\textbf{x} \quad f(v) = \int_{\Omega} fv \ d\textbf{x}.
\end{equation*}

\section{R-tree-based agglomeration}\label{sec:rtree}

We briefly summarize the basic properties of the R-tree data structure proposed by Guttman in the seminal paper~\cite{guttman::rtree} and discuss its variants following~\cite{manolopoulos::rtree}.

R-trees are hierarchical data structures used for the dynamic organization of collections of $d$-dimensional geometric objects, which they represent by minimum axis-aligned bounding rectangles (MBRs). In this context, \emph{dynamic} means that inserting or deleting tree elements requires no global reorganization. We use the terms MBR and \emph{bounding box} interchangeably throughout the paper. Each entry of an internal R-tree node identifies a child node and stores the MBR of all entries in that child.
The actual data is stored in the \emph{leaves}, i.e., the terminal nodes of the tree. We summarize these properties in the following definition.

\begin{definition}[R-tree of order $(m,M)$]
    An R-tree of order $(m,M)$ has the following characteristics:
    \begin{itemize}
        \item Each \emph{leaf} node contains between $m$ and $M$ entries, where $m \leq M/2$, except that the root may contain fewer entries if the total number of objects is less than $m$. Each entry is a pair (MBR, id), where id identifies the object and MBR is its minimum bounding rectangle;
        \item Each \emph{internal} node also contains between $m$ and $M$ entries. Each entry is a pair (MBR, p), where p points to a child node, which may be either an internal node or a leaf, and MBR bounds all entries in that child;
        \item All leaves are at the same level;
        \item The root contains at least two entries unless it is a leaf, in which case it may contain zero or one entry. This is the only exception to the requirement that each leaf contain at least $m$ entries.
    \end{itemize}
\end{definition}
Figure~\ref{fig:rtree_example} shows the minimum bounding rectangles of several geometric objects (not shown), while Figure~\ref{fig:rtree_diagram} shows the associated R-tree of order $(2,4)$. In this example, the leaf level stores the minimum bounding rectangles D, E, F, G, H, I, J, K, L, M, and N, while the internal node contains the three MBRs A, B, and C. The empty slots in the internal node and the rightmost leaf indicate that each can store one more entry because the tree has order $(2,4)$.

\begin{figure}[htb]

  \begin{minipage}[b]{0.5\textwidth}
  \centering
  
  \includegraphics{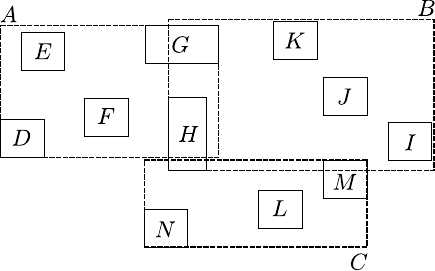}
  
  \caption{Example of the MBRs of several geometric objects (not shown).}\label{fig:rtree_example}
  
  \end{minipage}\hfill           
  \begin{minipage}[b]{0.5\textwidth}
  \centering
  \begin{forest}for tree={bplus, edge={->},l sep=1cm}
    [\mpn{A,B,C,\phantom{M}}
        [\mpn{D,E,F,G},multi edge=one]
        [\mpn{H,I,J,K},multi edge=two]
        [\mpn{L,M,N,\phantom{M}},multi edge=three]
    ]
  \end{forest}
  
  \caption{Corresponding R-tree data structure.}\label{fig:rtree_diagram}
  \end{minipage}
  
\end{figure}

The original R-tree is based solely on minimizing the measure of each MBR. Several variants have been proposed to improve performance or flexibility for particular applications.

For high-quality grid agglomeration, it is desirable to minimize the \emph{overlap} between MBRs. Larger overlaps increase the number of paths processed during queries. Moreover, reducing overlap makes the agglomerated elements conform more closely to their bounding boxes, so that the resulting agglomerated grid is qualitatively closer to a grid of rectangles or boxes. Among the available variants, we adopt the R\textsuperscript{*}-tree data structure introduced in~\cite{beckmann::rtree}. R\textsuperscript{*}-trees pursue the following criteria:
\begin{itemize}
    \item \emph{Minimization of the area covered by each MBR}. This criterion aims at minimizing the dead space (area covered by MBRs but not by the enclosed elements) to reduce the number of paths pursued during query processing;
    \item \emph{Minimization of overlap between MBRs}. The larger the overlap, the larger the expected number of paths followed for a query. As such, this criterion has the same objective as the previous one;
    \item \emph{Minimization of MBR perimeters}. More box-like bounding boxes reduce query times, which deteriorate in the presence of large overlaps or heterogeneous shapes. Moreover, compact objects are easier to pack, so the corresponding MBRs at higher levels are expected to be smaller;
    \item \emph{Maximization of storage utilization}. Increasing the storage utilization per node generally reduces query costs by reducing the height of the tree. This is especially important for large queries satisfied by a substantial fraction of the entries. Conversely, low storage utilization tends to require visiting more nodes during query processing.
\end{itemize}
These requirements may compete with one another. For example, keeping the area and overlap of MBRs small may reduce the number of entries packed into each tree node and hence lower storage utilization. The R\textsuperscript{*}-tree therefore uses heuristics to balance these criteria. We refer to the original paper~\cite{beckmann::rtree} for the relevant algorithmic details.

Our implementation relies on the \textsc{Boost.Geometry} module of the Boost C++ Libraries~\cite{boost} to construct and manipulate R\textsuperscript{*}-trees. Boost provides generic concepts, geometry types, and algorithms for computational geometry. Its interfaces are designed to be independent of the spatial dimension, coordinate system, and data type.

In the remainder of the paper, we do not distinguish between R-trees and R\textsuperscript{*}-trees, since we always employ the latter. For details on the implementation and validation of the R-tree agglomeration procedure for finite element meshes, we refer to~\cite{Feder::R3MG}.

\section{Multilevel preconditioner}\label{sec:preconditioner}
We now describe the procedure that builds degree-of-freedom (DoF) agglomerates and sets up the multigrid preconditioner for the conjugate gradient solver (CG)~\cite{Stiefel::CG}. We employ one multigrid V-cycle as a preconditioner, since this is known to be more robust than using it as a solver~\cite{HariStadlerBirosHighOrderMG}. In particular, we describe the strategies that we use to assemble the transfer matrices that allow us to build the multilevel hierarchy. We emphasize that the level operators $\{A_i\}_i$ we employ are generated through an \emph{inherited} approach~\cite{AntoniettiSartiVeraniDGhpMG}, i.e. it holds that
\begin{equation*}
    A_i = P_i^TAP_i,
\end{equation*}
where $P_i$ is the prolongation operator from the i-th level to the finest level and $P_i^T$ is the restriction operator. It follows that in order to obtain all the level operators, it suffices to define the prolongation operators $P_i$.

In~\cite{Feder::R3MG}, it has been shown that it is possible to agglomerate mesh elements through the R-tree data structure to build a nested sequence of meshes that naturally induces a multilevel hierarchy. To build the R-tree, MBRs of the mesh elements $\text{MBR}(\tau_i)$ serve as leaves, and mesh elements are agglomerated into polytopic elements. 

In order to build a full multilevel hierarchy in a conforming setting, additional care must be taken. Indeed, in the discontinuous Galerkin setting employed in~\cite{Feder::R3MG}, it is enough to simply build the transfers $P_i$ as the natural injection operator between discontinuous finite element spaces. The coarse finite element space on the polytopic element $K$ is defined by restriction. In particular, let $B_K$ be the bounding box of $K$: the coarse basis functions are defined by restricting the basis of $\mathcal{Q}^{p'}(B_K)$ to $K$. The nestedness of the grid hierarchy induces a nested sequence of finite element spaces. 

In our setting, however, we must account for the fact that there is no redundancy of support points located on the same vertices. To deal with the uniqueness of the interface support points, we have to distinguish between the following: 
\begin{itemize}
    \item building the prolongation from the finest level $M$ to any agglomerated level $l \in \{1, \dots, M-1\}$;
    \item building prolongations between consecutive agglomerated levels.
\end{itemize}
Let $\{\mathcal{T}_l\}_{l = 1}^M$ with $\mathcal{T}_M = \mathcal{T}_h$ be a sequence of nested agglomerated grids generated through the R-tree outlined in Section~\ref{sec:rtree}. Let $V^l_h$ be the finite element space on level $l \geq 1$, and let $V^M_h = V_h$. We can define $V_h^{M-1}$ through the restriction operator $P^T_{M-1}$.
In order to build the prolongation from level $M-1$ to level $M$, we can split the fine support points between the bounding boxes. In fact, to build the natural injection, only point values of the coarse basis functions are needed. In other words, the coarse basis functions are defined by interpolating the basis functions of $\mathcal{Q}^{p'}(B_K)$ with the basis functions of the fine elements $\tau_i \subset K$. If a support point $p_i$ of a basis function $\phi_i$ on $\tau_i$ lies in more than one agglomerate $K$, then $\phi_i$ may be used to interpolate the basis of \emph{only one} of the $\mathcal{Q}^{p'}(B_K)$ for which $p_i \in B_K$. Concretely, to build the injection operator we:
\begin{enumerate}
    \item visit each agglomerate $K$ in the order they are listed in the array when extracted;
    \item build a local interpolation operator on $\mathcal{Q}^{p'}(B_K)$ using only the fine basis functions whose support points in $K$ have not already been used;
    \item distribute the local interpolation operator to the global injection matrix according to the DoF numbering.
\end{enumerate}
We provide a visualization of the procedure we use to uniquely assign support points in Figure~\ref{fig:bbox_grid} with $p = p' = 1$. 

\begin{figure} 
  \centering
  \includegraphics[width=0.5\textwidth]{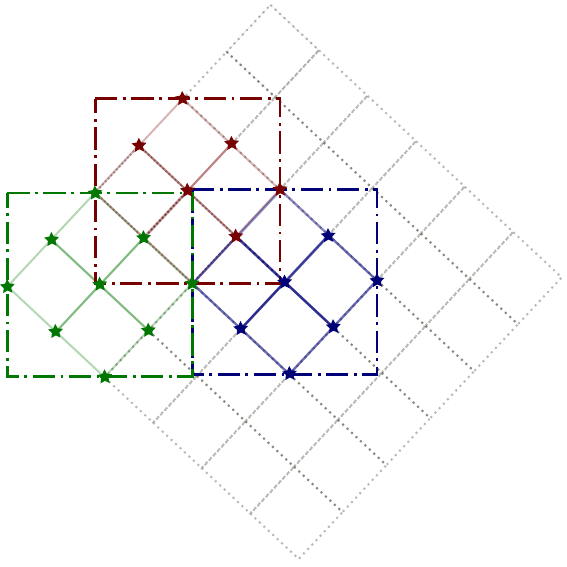}
  \caption{Assignment of support points of the basis functions. If the order of the agglomerates (bounding boxes with dash-dot line) in the array is green, red, blue we first visit the green agglomerate and assign all the support points to it (marked by a green star). Then, the red agglomerate can only own the points marked with a red star. In the same way the blue agglomerate can only own the points marked with a blue star.}
  \label{fig:bbox_grid}
\end{figure}

Prolongation operators, between the agglomerated levels, are built in the same way as in~\cite{Feder::R3MG} since there is again redundancy of support points. If we call $\tilde{P}_j^{j+1}$ the prolongation operator between agglomerated levels $j$ and $j+1$, with $j \leq M - 2$, we can define $P_i$ as 
\begin{equation*}
    P_i = P_{M-1}\tilde{P}_{M-2}^{M-1}\cdots\tilde{P}_{i+1}^{i+2}\tilde{P}_{i}^{i+1}.
\end{equation*}
It follows that, if $\{\phi_j \}_{j=1}^{n_h}$ are the basis functions of $V_h$ then $P_l \in \mathbb{R}^{n_h \times n_l}$ and
\begin{equation*}
    V^l_h = \left\{\sum_{j=1}^{n_h} (P_l c)_j \phi_j : c \in \mathbb{R}^{n_l}\right\}.
\end{equation*}
With these definitions, we obtain that $V^1_h \subset V^2_h \subset \dots \subset V^{M-1}_h \subset V_h$, i.e. the spaces are nested. We refer to the agglomeration strategy described above as \emph{cell-based agglomeration}. We report in Figure~\ref{fig:partitioning} a visualization of the agglomerates and the associated bounding boxes at a fixed level.

\begin{figure} 
  \centering
  \includegraphics[width=0.5\textwidth]{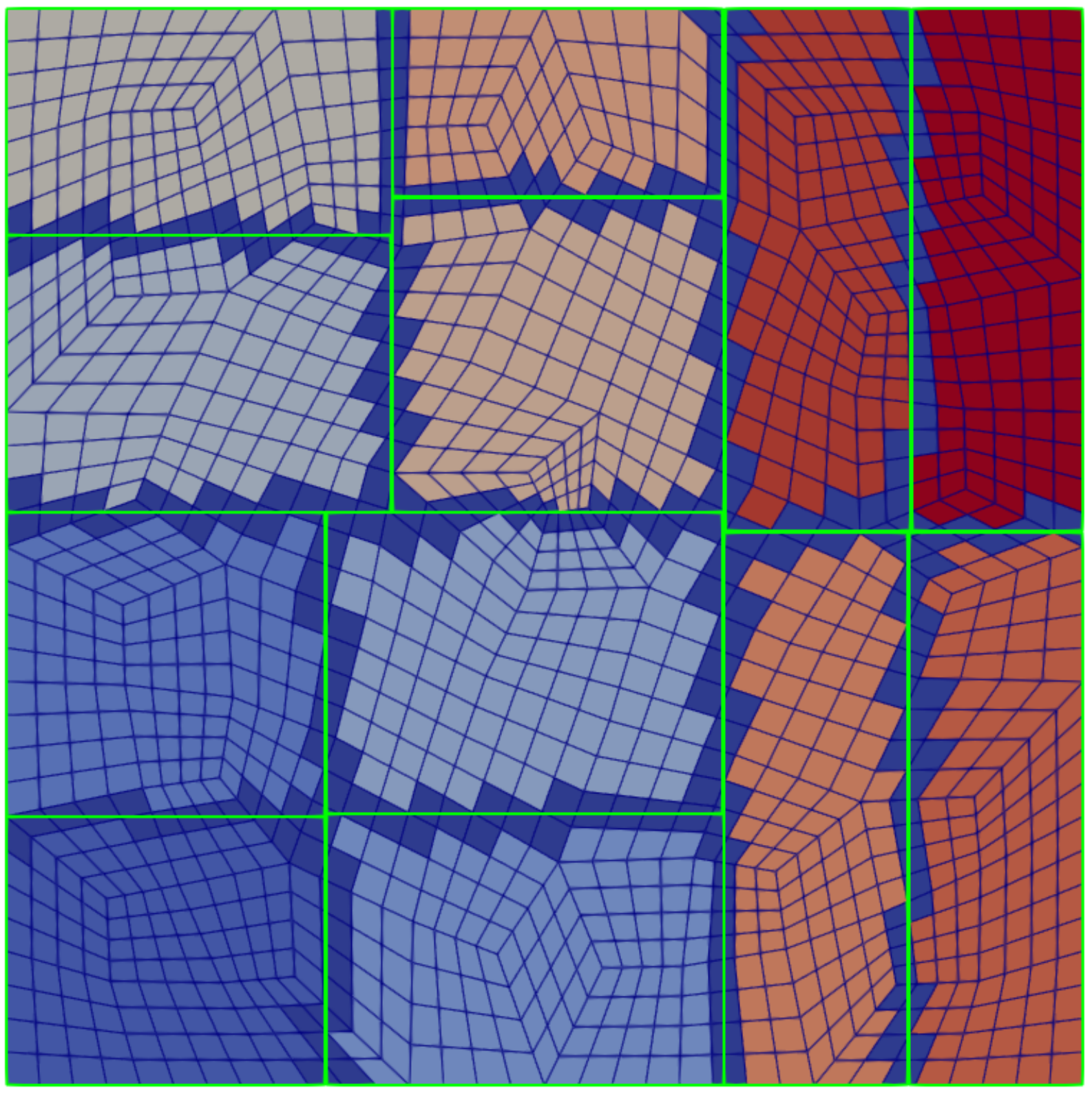}
  \caption{Agglomerates of the mesh elements obtained through the R-tree. We report in dark-blue the overlap between agglomerates: the support points of those mesh elements are split between agglomerates. Each of the overlap mesh elements is assigned to more than one agglomerate, each support point instead is uniquely assigned to one agglomerate. In green we highlight the bounding boxes of the agglomerates.}
  \label{fig:partitioning}
\end{figure}

\begin{remark}\label{rmrk:parameters}
It is important that each agglomerate contains enough support points to build the local interpolation operator. If there are not enough support points, the sub-level matrices will be singular. The parameters $(m,M)$ of the R-tree must be tuned accordingly to ensure that the local interpolations are successful. The correct values for $m$ and $M$ also depend on the polynomial degrees $p$ and $p'$. To increase the robustness of the procedure, we can increase $m$ or skip the leaves level of the R-tree in the agglomeration. We usually choose the latter, but both strategies are viable.
\end{remark}

In order to build the transfer operators $P_i$, we only need point evaluations of the basis functions. This suggests another agglomeration strategy: given the mesh $\mathcal{T}_h$, we can agglomerate using an R-tree whose leaves are the support points of the basis functions of $V_h$. The generated agglomerates are composed of points. We can then build the transfers $P_i$ in exactly the same way as for cell-based agglomeration. The only difference is that there is no need to assign shared support points. In fact, the R-tree identifies the points uniquely owned by each bounding box $B_K$ of a point agglomerate $K_p$. We call this agglomeration strategy \emph{point-based agglomeration}. We emphasize that the shapes of agglomerates produced by point-based agglomeration can differ significantly from those produced by cell-based agglomeration.

\begin{remark}
R-tree-based agglomeration is a heuristic procedure: the mi\-ni\-mi\-za\-tion/ma\-xi\-mi\-za\-tion criteria it tries to satisfy are not guaranteed to hold. It might happen that, even though an agglomerate contains enough support points, their positions in the bounding box do not allow local interpolation. We usually do not observe this phenomenon since the heuristics try to create agglomerates of good geometric quality (we refer to~\cite{Feder::R3MG} for an extensive validation of the R-tree-based agglomeration procedure).
\end{remark}

\section{Convergence analysis of the two-level preconditioner}\label{sec:analysis}
In this section, we show the convergence properties of the two-level version of our multilevel preconditioner. The convergence analysis is based on the two-level subspace-correction framework described in~\cite{xu:review} by Xu and Zikatanov. We follow the general approach to the construction of a coarse space. The core idea of the analysis is to extend the Nicolaides coarse space~\cite{Nicolaides1987Deflation,victorita::DD} to higher order polynomials instead of just constants. We first show that if we employ the Nicolaides coarse space (i.e. $p \geq 1, p' = 0$), the multigrid algorithm converges. Then, we generalize to higher order polynomial degrees.

For ease of readability, from now on we drop the subscript $h$ from the finite element space $V_h$.
Let $V'$ be the usual dual space of $V$. Then
\begin{itemize}
    \item We consider the classical piecewise Lagrange polynomials as a basis for $V$;
    \item $B_i$, for $i = 1, \dots,  J$, are the bounding boxes produced by the R-tree at the coarse level;
    \item $\mathcal{P}_i  = \{p_1^i, \dots,  p^i_{l_i} \} \subset B_i$ are the support points of the basis functions of $V$ that are uniquely assigned to the box $B_i$ as previously described;
    \item $R : V' \to V$ is a converging diagonal smoother. We set $R = \omega D^{-1}$, where $D = \text{diag}(A)$ is the diagonal of the stiffness matrix, i.e. $R$ is the damped Jacobi smoother. Let $\bar{R} = R + R' - R'AR$ where $R'$ is the adjoint of $R$. It is well-known that the damped Jacobi smoother is converging if and only if $0 < \omega < 2/\rho(D^{-1}A)$. During the analysis, as explained in~\cite{xu:review}, we will replace $\bar{R}^{-1}$ with a simpler but spectrally equivalent symmetric positive definite (SPD) operator $D = \text{diag}(A)$.
\end{itemize}
\begin{remark}
    The smoother we use in the numerical experiments is not the classical Jacobi diagonal smoother. We use the Chebyshev acceleration of the Jacobi method~\cite{ADAMS2003593}, i.e. we are replacing $R$ with an appropriate Chebyshev polynomial of degree three $\text{Cheb}_3(R)$. We will develop the theory for the Jacobi smoother: in the experiments we are just using the accelerated version of the Jacobi iteration.  We will not provide explicit convergence estimates in terms of the Chebyshev smoother: estimates in terms of the Jacobi smoother will provide a sufficient proxy. 
\end{remark}
Following the construction in~\cite{victorita::DD}, we define an overlapping partition of $\Omega$
\begin{equation}
    \Omega = \bigcup_{i=1}^J \Omega_i, \quad \Omega_i = \bigcup_{\tau_j \ : \ \mathcal{P}_i\cap \bar{\tau}_j \neq \emptyset} \tau_j .
\end{equation}
Let $H_i = \text{diam}(\Omega_i)$ and $H = \max_{i=1,\dots,J} H_i$. We define spaces $V_1, V_2, \dots, V_J$ as
\begin{equation*}
    V_i = \Span_{j = 1, \dots, l_i} \{\phi_j^i \},
\end{equation*}
where $\phi_j^i$ is a basis function of an element $\tau_{j'} \subset \Omega_i$ such that $\phi_j^i(p_j^i) = 1$ for $p_j^i \in \mathcal{P}_i$. From the definition of the basis function, we have that for each $v_i \in V_i$, $\text{supp}(v_i) \subset \Omega_i$. Each space $V_i$ is not necessarily a subspace of $V$ (due to boundary conditions), but is tied to $V$ by $I_h^V$, the standard interpolation operator on $V$. Following the notation in~\cite{xu:review}, since $I_h^V (V_i) \subset V$, we define $\Pi_i : V_i \to V$ for every $i$ as 
\begin{equation*}
\Pi_i = I_h^{V}.
\end{equation*} 
We emphasize that the main role of the interpolation $I_h^V$ is to enforce the boundary conditions.

For each $V_i$ we need to define a coarse space $V_i^c \subset V_i$. Let $I_h^{V_i}$ be the local interpolation operator on $V_i$, i.e.
\begin{equation}\label{eq:localinterpolation}
    I_h^{V_i}v = \sum_{j=1}^{l_i} v(p_j^i)\phi_j^i .
\end{equation}
We define the local coarse space for $i=1,\dots,J$ as
\begin{equation}
    V_i^c = \Span \Big\{I_h^{V_i}(\phi^{DG}_{i,q}) : \phi_{i,q}^{DG} \in \mathcal{Q}^{p'}(B_i) \Big\},
\end{equation}
where $p' \geq 0$ and $\{\phi_{i,q}^{DG}\}_q$ is a basis of $\mathcal{Q}^{p'}(B_i)$. Let $\widetilde{W} = V_1 \times V_2 \times \dots V_J$ be a space with scalar product
\begin{equation*}
    (\tilde{u}, \tilde{v}) = \sum_{i=1}^J (u_i, v_i),
\end{equation*}
where $(\cdot, \cdot)$ is the standard $L^2$ scalar product, $\tilde{u} = (u_1, \dots , u_J)^T$, $\tilde{v} = (v_1, \dots , v_J)^T$, $u_i, v_i \in V_i$. Let $\Pi_W : \widetilde{W} \to V$ be
\begin{equation*}
    \Pi_W \tilde{u} = \sum_{i=1}^J \Pi_i u_i.
\end{equation*}
Let $\Pi_i'$ be the adjoint of $\Pi_i$. We can define local SPD operators $A_i : V_i \to V_i'$ as 
\begin{equation*}
    A_i = \Pi_i'A\Pi_i .
\end{equation*}
We also define $\widetilde{A}_W = \text{diag}(A_1, A_2, \dots, A_J)$ and the smoother restrictions $D_i : V_i \to V_i'$ as follows
\begin{equation*}
    D_i = \Pi_i' D  \Pi_i .
\end{equation*}
Again, we define $\widetilde{D} = \text{diag}(D_1, D_2, \dots, D_J)$. Let $Q_i : V_i \to V_i^c$ be the local orthogonal projection with respect to the scalar product $(\cdot, \cdot)_{D_i}$ defined as
\begin{equation*}
    (u_i, v_i)_{D_i} = (D_iu_i,v_i) .
\end{equation*}

\begin{assumptions}\label{xu_core_assumptions}
    \begin{enumerate}
        \item For all $\tilde{w} \in \widetilde{W}$, it holds that 
        \begin{equation}\label{eq:smootherstab}
            \norm{\Pi_W \tilde{w}}_D^2 \leq C_{1}\norm{\tilde{w}}^2_{\widetilde{D}} .
        \end{equation}
        \item For all $u \in V$ there exists $\tilde{w} \in  \widetilde{W}$ with $u = \Pi_W\tilde{w}$ such that
        \begin{equation}\label{eq:stabledec}
            \norm{\tilde{w}}^2_{\widetilde{A}_W} \leq C_2 \norm{u}^2_A .
        \end{equation}
    \end{enumerate}
\end{assumptions}
\noindent Finally, we define, for each local space $V_i^c$, the ``conditioning'' of the coarse space as follows
\begin{equation}\label{eq:condcoarse}
    \mu_i^{-1}(V_i^c) = \max_{v_i \in V_i} \min_{v_i^c \in V_i^c} \frac{\norm{v_i - v_i^c}_{D_i}^2}{\norm{v_i}^2_{A_i}} = \max_{v_i \in V_i} \frac{\norm{v_i - Q_iv_i}_{D_i}^2}{\norm{v_i}^2_{A_i}}.
\end{equation}
We quote the following result from~\cite{xu:review}:
\begin{theorem} \label{xu:core_theo}
    Let $\mu_c = \min_{i = 1, \dots, J} \mu_i(V_i^c)$. If Assumptions \ref{xu_core_assumptions} hold then the two-level algebraic multigrid algorithm whose coarse space is $V^c = \sum_{i=1}^J \Pi_i V_i^c$ converges with rate
    \begin{equation}\label{eq:errpropagation}
        \norm{E}_A^2 \leq 1 - \frac{\mu_c}{C_1 C_2 c^D},
    \end{equation}
    where $c^D$ is such that
    \begin{equation*}
        \norm{v}^2_{\bar{R}^{-1}} \leq c^D \norm{v}^2_{D} \quad \forall v \in V.
    \end{equation*}
\end{theorem}
\noindent To estimate the convergence rate of the two-level AMG we need to provide estimates of the terms $c^D,\mu_c, C_1, C_2$.
\begin{remark}
   Let $i_c : V^c \to V$ be the injection of $V^c \subset V$ into $V$. We can write $E$ explicitly:
    \begin{equation*}
        E  = (I  - RA)(I - i_{c}A_c^{-1}i_c'A),
    \end{equation*}
    where $A_c = i_{c}'Ai_c$.
\end{remark}

\begin{lemma} \label{lemma:smoothers_equivalence}
    Let $R = \omega D^{-1}$ and $D = \text{diag}(A)$. Then, there exists a constant $c^D > 0$ such that
    \begin{equation*}
        \norm{v}^2_{\bar{R}^{-1}} \leq c^D \norm{v}^2_{D} \quad \forall v \in V.
    \end{equation*}
    The explicit expression of $c^D$ is
    \begin{equation*}
        c^D = \frac{1}{\omega (2 - \omega \rho)}, 
    \end{equation*}
    where $\rho = \rho(D^{-1}A)$ is the spectral radius of $D^{-1}A$.
\end{lemma}
\begin{proof}
    Let $M = D^{-1/2}AD^{-1/2}$ with eigenvalues $\theta_1 \leq \dots \leq \theta_n$. The $\theta_i$-s are the eigenvalues of $D^{-1}A$. Let $\rho = \rho(D^{-1}A) = \theta_n$. Since $D^{-1} = D^{-1/2}D^{-1/2}$, and $A = D^{1/2}MD^{1/2}$, we have that
    \begin{equation*}
        \bar{R} = \omega D^{-1/2}\left(2I - \omega M \right)D^{-1/2}.
    \end{equation*}
    Since $0 < \omega < 2/\rho$ by assumption, we have that $2I - \omega M$ is SPD. Then, 
    \begin{equation*}
        \bar{R}^{-1} = \frac{1}{\omega}D^{1/2}\left(2I - \omega M \right)^{-1}D^{1/2}.
    \end{equation*} 
    If we diagonalize $M = U\Lambda U^{T}$, for any vector $v \in \mathbb{R}^n$ we set $z = U^TD^{1/2}v$ and obtain that 
    \begin{equation*}
        \norm{v}^2_{\bar{R}^{-1}} = \frac{1}{\omega}\sum_{i=1}^n \frac{z_i^2}{2 - \omega \theta_i}. 
    \end{equation*}
    Since $f(\theta) = \frac{1}{2 - \omega \theta}$ is increasing in $\theta$ we have that 
    \begin{equation*}
        \norm{v}^2_{\bar{R}^{-1}} \leq \frac{1}{\omega(2 - \omega \rho)}\sum_{i=1}^n z_i^2 = \frac{1}{\omega(2 - \omega \rho)}\norm{v}^2_D.
    \end{equation*}
\end{proof}
We define the overlap and the interior of each agglomerate as follows: let $S_j$ be the set of support points of the fine mesh element $\tau_j$. For each $\Omega_i$ we define
    \begin{itemize}
    \item $\Omega_{i,h} = \bigcup_{j : S_j \not\subset \mathcal{P}_i, \ \tau_j \subset \Omega_i} \tau_j$, i.e. the overlap;
    \item $\Omega_i^{\circ} = \Omega_i \setminus \Omega_{i,h}$ the internal elements of each agglomerate.
\end{itemize} 
Let us impose the following assumptions on the agglomeration hierarchy for the rest of the analysis:
\begin{assumptions}[Admissible hierarchy] \label{admissible_hierarchy}
\begin{enumerate}
\item Each coarse space is built from tensor-product polynomials on the axis-aligned bounding box $B_i$ of the agglomerate. We assume that the aspect ratios of these boxes are uniformly bounded and that the interior region $\Omega_i^\circ$ occupies a nondegenerate portion of $B_i$. We also assume that $\Omega_i^\circ$ is connected, and satisfies Poincaré and polynomial approximation estimates with constants independent of $h$, $H_i$, $p$, and $p'$ up to the standard dependence on the uniform box aspect-ratio bound.
\item We assume that the boundary layer $\Omega_{i,h}$ consists of a uniformly bounded number of fine-cell layers, hence has thickness $O(h)$, and that the overlap of these layers is uniformly bounded.
\item We assume that the boundary $\partial\Omega_{i,h}$ is uniformly shape-regular and its ($d-1$)-dimensional measure scales like $H^{d-1}$.
\item We assume $p'\le p$ and that the local interpolation
from $\mathcal Q^{p'}(B_i)$ is uniformly stable. This can be enforced by using stable tensor-product interpolation nodes on the Cartesian box $B_i$; if the owned support points of the agglomerate are used instead, one must assume that they are uniformly unisolvent and stable for interpolation on $B_i$.
\end{enumerate}
\end{assumptions}
These assumptions are consistent with the R-tree construction used in the numerical experiments: the R-tree ownership rule gives the algebraic decomposition and partition of unity, while the box-regularity, nondegeneracy, and interpolation assumptions exclude degenerate agglomerates for which uniform approximation estimates cannot hold.

From now on, let $a\lesssim b$ mean that $a \leq C b$ with $C$ a constant that does not depend on discretization parameters ($H$ or $h$) or polynomial degrees (we remark that every time we use $\lesssim$ the constant $C$ can change). Let us also impose that $a \sim b$ means that $a \lesssim b$ and $b \lesssim a$.

For this part of the analysis, let us restrict to the case $\mathcal{R}^p = \mathcal{Q}^p$, i.e. quadrilateral or hexahedral elements and consider as basis functions of $V$ the classical tensor-product Lagrange basis functions with Gauss-Lobatto-Legendre (GLL) nodes. To estimate $C_1, C_2$ and $\mu_c$ we first need to prove a few technical lemmas. 
\begin{lemma}\label{app:shape_fun_p_dep}
Let $\hat{Q} = [0,1]^{d=2,3}$ be a reference square or cubic element. Let $\phi_{\mathbf{j}} \in \mathcal{Q}^p(\hat{Q})$ be the tensor-product Lagrange basis functions with GLL nodes on $\hat{Q}$. Let $p \geq 1$, $\mathbf{j} = (j_1, \dots, j_d)$ the multi-index used to enumerate the Lagrange basis function on the reference element. It holds that 
\begin{equation*}
p^{-2d} \lesssim \norm{\phi_{\mathbf{j}}}^2_{L^2(\hat{Q})} \lesssim p^{-d}.
\end{equation*}
\end{lemma}
\begin{proof}
First of all let us consider the one-dimensional Lagrange polynomials of degree $p$ on the interval $[-1,1]$, $\{\psi_j(x)\}_{j=0}^p$ with GLL nodes $\{\xi_j\}_{j=0}^p$. From~\cite{Canuto2006} we know that
\begin{equation*}
\psi_j(x) = \frac{1}{p(p+1)}\frac{1 - x^2}{\xi_j - x}\frac{L'_p(x)}{L_p(\xi_j)}, \quad j=0,\dots,p,
\end{equation*}
where $L_p, L_p'$ are the standard Legendre polynomial of degree $p$ on $[-1,1]$ and its derivative. We shall prove that 
\begin{equation}
\norm{\psi_j}^2_{L^2(-1,1)} = \frac{2p}{2p + 1}w_j, \quad j=0,\dots,p,
\end{equation}
where $w_j$ is the $j$-th weight of the corresponding GLL quadrature, i.e.
\begin{equation*}
w_j = \frac{2}{p(p+1)}\frac{1}{[L_p(\xi_j)]^2}, \quad j=0,\dots,p.
\end{equation*} 
Since we know that the GLL quadrature formula is exact up to polynomials of degree $2p - 1$, for $k = 0,1,\dots,p-1$ we have that 
\begin{equation*}
 \int_{-1}^1 \psi_j(x)L_k(x) dx = \sum_{i=0}^p w_i \psi_j(\xi_i)L_k(\xi_i) = w_j L_k(\xi_j).
\end{equation*}
It follows that, if we expand in the Legendre basis $\psi_j = \sum_{k=0}^p c_k^{(j)}L_k$ we have that 
\begin{equation*}
c_k^{(j)} = \frac{1}{\norm{L_k}^{2}_{L^2(-1,1)}}\int_{-1}^1 \psi_j(x)L_k(x) dx = \frac{2k + 1}{2}w_j L_k(\xi_j), \quad k=0,\dots,p-1,
\end{equation*}
where we have used the fact that $\norm{L_k}^{2}_{L^2(-1,1)} = 2/(2k + 1)$ (see, for instance~\cite{Canuto2006}).
Now, consider the polynomial $q_j(x) = \frac{(1-x^2)}{\xi_j - x}L'_p(x)$. The leading coefficient of $q_j(x)$ is $-p\alpha_p$ where $\alpha_p$ is the leading coefficient of $L_p(x)$. Then, by expanding $q_j$ in the Legendre basis, we obtain that the Legendre coefficient of $L_p$ in $q_j$ is $-p$. Then,
\begin{equation*}
\int_{-1}^1  q_j(x)L_p(x) dx = -p \norm{L_p}^2_{L^2(-1,1)}  = \frac{-2p}{2p + 1},
\end{equation*}
and by substitution we obtain that 
\begin{equation}\label{last_expansion_term}
c_p^{(j)} = \frac{2p + 1}{2} \int_{-1}^1 \psi_j(x) L_p(x) dx = \frac{2p +1 }{2}\frac{-1}{p(p+1)L_p(\xi_j)}\frac{-2p}{2p + 1} = \frac{1}{(p+1)L_p(\xi_j)}.
\end{equation}
From Parseval's identity we know that 
\begin{equation*}
\norm{\psi_j}^2_{L^2(-1,1)} = \sum_{k=0}^p \frac{2}{2k + 1}[c_k^{(j)}]^2 = w_j^2\sum_{k=0}^{p-1} \frac{2k + 1}{2}L_k(\xi_j)^2 + \frac{2}{2p + 1}[c_p^{(j)}]^2 .
\end{equation*}
From the \emph{Christoffel-Darboux formula} (see e.g. Corollary 3.3 in~\cite{ShenTangWang2011}) we have that 
\begin{equation*}
\sum_{k=0}^{p-1} \frac{2k + 1}{2}L_k(\xi_j)^2 = \frac{p}{2}[L_p'(\xi_j)L_{p-1}(\xi_j) - L_{p-1}'(\xi_j)L_p(\xi_j)] := K_{p-1}(\xi_j).
\end{equation*}
Now, let us prove that
\begin{equation}\label{darboux_formula}
K_{p-1}(\xi_j) = \frac{p^2}{2}[L_p(\xi_j)]^2.
\end{equation}
\begin{itemize}
\item If $j=0,p$, i.e. $\xi_j = \pm 1$, then from the fact that $L_p(\pm 1) = (\pm 1)^p$ and $L_p'(\pm 1) = \frac{1}{2}(\pm 1)^{p-1}p(p+1)$ (see e.g.~\cite{ShenTangWang2011}) it is straightforward to prove that
\begin{equation*}
K_{p-1}(\pm 1) = \frac{p^2}{2};
\end{equation*}
\item If $j=1,\dots,p-1$, by combining the derived three term recurrence and the derivative recurrence relation of Legendre polynomials (see e.g.~\cite{ShenTangWang2011})
\begin{align*}
(p+1)L'_{p+1}(x) &= (2p+1)[L_p(x) + xL'_p(x)] - pL'_{p-1}(x), \\ (2p+1)L_p(x) &= L'_{p+1}(x) - L'_{p-1}(x),
\end{align*}
and the fact that $L'_p(\xi_j) = 0$ we obtain that 
\begin{equation*}
L'_{p-1}(\xi_j) = -pL_p(\xi_j).
\end{equation*}
It follows that
\begin{equation*}
K_{p-1}(\xi_j) = \frac{p^2}{2}[L_p(\xi_j)]^2.
\end{equation*}
\end{itemize}
Combining (\ref{darboux_formula}) with the formula of the weights $w_j$ we obtain that
\begin{equation}\label{final_darboux}
K_{p-1}(\xi_j) = \frac{p^2}{2}\frac{2}{p(p+1)w_j} = \frac{p}{(p+1)w_j}.
\end{equation}
Finally, using (\ref{last_expansion_term}) and (\ref{final_darboux}), we obtain
\begin{align*}
\norm{\psi_j}^2_{L^2(-1,1)} &= w_j^2 \frac{p}{(p+1)w_j} + \frac{2}{2p + 1}\frac{1}{(p+1)^2[L_p(\xi_j)]^2} \\
&= w_j \frac{p}{(p+1)} + \frac{1}{2p + 1}\frac{w_j p}{(p+1)} \\
&= \frac{2p}{2p+1}w_j.
\end{align*}

Let us fix $\phi_{\mathbf{j}}(\mathbf{x}) = \prod_{i=1}^d \psi_{j_i}(x_i)$, where $j_i = 0, \dots, p$ and $\mathbf{x} = (x_1, \dots, x_d)$. Then, we have that 
\begin{equation*}
\norm{\phi_{\mathbf{j}}}^2_{L^2(\hat{Q})} = \int_{\hat{Q}} \phi_{\mathbf{j}}(\mathbf{x})^2 d\mathbf{x} = \prod_{i=1}^d \int_0^1 \psi_{j_i}(x_i)^2 dx_i \sim \prod_{i=1}^d w_{j_i},
\end{equation*}
where in the last inequality we have rescaled the integration domain from $[0,1]$ to $[-1,1]$.
Finally, we obtain the thesis by observing that 
\begin{equation*}
    p^{-2d} \lesssim \prod_{i=1}^d w_{j_i} \lesssim p^{-d},
\end{equation*}
since $p^{-2} \lesssim w_{j_i} \lesssim p^{-1}$ (see e.g. lemma 3.3 in~\cite{Parter1999}).
\end{proof}
We need another technical lemma before we can get into the analysis of the two-level preconditioner:
\begin{lemma}\label{lemma:trace_inequality}
    Let $u \in V$ and 
    \begin{equation*}
        \bar{u}_i = \frac{1}{|\Omega_i|}\int_{\Omega_i} u \ dx, \quad i=1,\dots,J.
    \end{equation*}
    Consider two overlapping agglomerates $\Omega_i$ and $\Omega_j$ and let $\Omega_{ij} = \Omega_i \cap \Omega_j$. Let $\Gamma_{ij} = \partial \Omega_{ij}$, if $|\Gamma_{ij}| \sim H^{d-1}$ then it holds that
    \begin{equation*}
        |\bar{u}_i - \bar{u}_j|^2H^{d-1} \lesssim H_i|u|^2_{H^1(\Omega_i)} + H_j|u|^2_{H^1(\Omega_j)}.
    \end{equation*}
\end{lemma}
\begin{proof}
    By the multiplicative trace inequality (see, for instance,~\cite{BrennerScott2008}) we have that
    \begin{align*}
        \norm{u - \bar{u}_i}_{L^2(\Gamma_{ij})}^2 &\lesssim \norm{u - \bar{u}_i}_{L^2(\Omega_{ij})}\norm{u - \bar{u}_i}_{H^1(\Omega_{ij})} \lesssim \norm{u - \bar{u}_i}_{L^2(\Omega_i)}\norm{u - \bar{u}_i}_{H^1(\Omega_i)} \\ &\lesssim H_i(1 + H_i)|u|_{H^1(\Omega_i)}^2 \lesssim H_i|u|_{H^1(\Omega_i)}^2,
    \end{align*}
    where we used the Poincaré inequality. The same holds for $\Omega_j$. Then, 
    \begin{align*}
        |\bar{u}_i - \bar{u}_j|^2H^{d-1} &\lesssim |\bar{u}_i - \bar{u}_j|^2|\Gamma_{ij}| \lesssim \norm{u - \bar{u}_i}_{L^2(\Gamma_{ij})}^2 + \norm{u - \bar{u}_j}_{L^2(\Gamma_{ij})}^2 \\ &\lesssim H_i|u|^2_{H^1(\Omega_i)} + H_j|u|^2_{H^1(\Omega_j)}.
    \end{align*}
\end{proof}

In our setting, to estimate $C_1$ we prove the following lemma:
\begin{lemma}
For any $\tilde{w} \in \widetilde{W}$ it holds that
\begin{equation}\label{eq:C1}
    \norm{\Pi_W \tilde{w}}_D^2 \leq \norm{\tilde{w}}_{\tilde{D}}^2,
\end{equation}
\end{lemma}
\begin{proof}
    Let $\mathcal N_h^0$ be the set of all support points of $V$ (including boundary support points). From the ownership rule of support points, we have that $\mathcal N_h^0=\dot\bigcup_{i=1}^J \mathcal P_i$.
    Let  $\tilde{w} = (v_1, \dots, v_J) \in \widetilde{W}$. Since $\Pi_i$ is the nodal interpolation operator $I_h^V$ we have that $(\Pi_i v_i)(p_k) = v_i(p_k)$ if $p_k \in \mathcal{P}_i$ and zero otherwise. Then, we have that
    \begin{equation*}
        (\Pi_W \tilde{w})(p_k) = \sum_{i=1}^J (\Pi_i v_i)(p_k) = \sum_{i=1}^J I_h^V v_i(p_k) = \begin{cases}
            v_i(p_k) & \text{if } p_k \in \mathcal{P}_i \setminus \partial \Omega, \\
            0 & \text{otherwise},
        \end{cases}
    \end{equation*}
    since $I_h^Vv_i(p_k) = 0$ for $p_k \in \partial \Omega$.
    Now we can finally show that
    \begin{equation*}
        \norm{\Pi_W \tilde{w}}_D^2 = \sum_{p_k \in \mathcal{N}_h^0} D_{kk} (\Pi_W \tilde{w})(p_k)^2 \leq \sum_{i=1}^J \sum_{p_k \in \mathcal{P}_i} D_{kk} v_i(p_k)^2 = \sum_{i=1}^J \norm{v_i}^2_{D_i} = \norm{\tilde{w}}_{\tilde{D}}^2.
    \end{equation*}
\end{proof}
\noindent It follows that $C_1 = 1$.

\begin{remark}[Partition of unity]
The ownership rule for support points has an important algebraic
consequence. Let $\mathcal N_h^0$ denote the set of all the
support points of $V$ (including boundary support points). By construction,
\[
  \mathcal N_h^0=\dot\bigcup_{i=1}^J \mathcal P_i .
\]
Hence, for every $u\in V$, the function $I_h^{V_i}u$ has the same
nodal values as $u$ on the support points in $\mathcal P_i$, and
zero nodal values on the support points owned by the other agglomerates. Since $u \in V$ it naturally follows that $I_h^{V_i} u \in V$ and the boundary conditions are satisfied.
Therefore, if we set $u_i = I_h^{V_i} u$ then,
\begin{equation}\label{eq:nodal_decomposition}
  \sum_{i=1}^J I_h^{V_i}u = \sum_{i=1}^J u_i = u
  \qquad \forall u\in V .
\end{equation}
Equivalently, the operators $I_h^{V_i}$ provide a partition of the
unconstrained degrees of freedom. In particular, if
$\{\phi_{i,q}^{DG}\}_q$ is a basis of $\mathcal Q^{p'}(B_i)$
satisfying $\sum_q\phi_{i,q}^{DG}=1$ on $B_i$, then the associated
coarse functions
\[
  \phi_{i,q}^C := I_h^{V_i}\phi_{i,q}^{DG}
\]
satisfy
\begin{equation}\label{eq:coarse_partition_unity}
  \sum_{i=1}^J \sum_q \phi_{i,q}^C
  =
  \sum_{i=1}^J I_h^{V_i}1
  =
  \mathbf 1_h = \mathbf 1,
\end{equation}
where $\mathbf 1_h$ denotes the nodal function that is equal to one on
the unconstrained degrees of freedom. Since no boundary
conditions are imposed on the spaces $V_i$, then $\mathbf 1_h$ is the constant function $\mathbf 1$.
\end{remark}

To estimate $C_2$, we first note that for every $u \in V$ we have
\begin{equation}
    u = \Pi_W \tilde{w}
\end{equation} 
where $\tilde{w} = (I_h^{V_1}u, \dots, I_h^{V_J}u)^T$. Then, let us prove that 
\begin{equation}\label{eq:C2}
    \norm{\tilde{w}}^2_{A_w} \lesssim p^{d+4}\frac{H}{h}\Bigl(1 + L^2\frac{H}{h} \Bigr)  \norm{u}^2_A,
\end{equation}
where $C$ does not depend on $h, H, p$ and $L = \max_{i=1,\dots,J} l_i$. In fact, it holds that
\begin{align*}
    \norm{\tilde{w}}^2_{A_w} &= \sum_{i=1}^J (A_i u_i, u_i) = \sum_{i=1}^J (A \Pi_i u_i, \Pi_i u_i) \\ &= \sum_{i=1}^J \norm{u_i}^2_A = \sum_{i=1}^J |u_i|^2_{H^1(\Omega_i)} \\ &\lesssim p^{d+4}\frac{H}{h}\Bigl(1 + L^2\frac{H}{h} \Bigr)|u|_{H^1(\Omega)}^2 = p^{d+4}\frac{H}{h}\Bigl(1 + L^2\frac{H}{h} \Bigr) \norm{u}^2_A,
\end{align*}
where the last inequality follows from the following lemma:
\begin{lemma}\label{lemma:stable_decomposition}
    For any $u \in V$ it holds that
    \begin{equation*}
        \sum_{i=1}^J |I_h^{V_i} u|^2_{H^1(\Omega_i)} \lesssim p^{d+4}\frac{H}{h}\Bigl(1 + L^2\frac{H}{h} \Bigr)|u|_{H^1(\Omega)}^2.
    \end{equation*}
\end{lemma}
\begin{proof}
    From (\ref{eq:coarse_partition_unity}) we have that 
    \begin{equation*}
        I_h^{V_i} u = I_h^{V_i}(u - \bar{u}_i) + \bar{u}_i I_h^{V_i}1.
    \end{equation*}
    It holds that 
    \begin{align*}
    |I_h^{V_i} (u - \bar{u}_i)|^2_{H^1(\Omega_i)} &\lesssim \Big| \sum_{j=1}^{l_i} [u(p_j^i) - \bar{u}_i] \phi_j^i \Big|^2_{H^1(\Omega_i)} \lesssim \norm{u - \bar{u}_i}^2_{L^{\infty}(\Omega_i)}l_i\sum_{j=1}^{l_i} |\phi_j^i|^2_{H^1(\Omega_i)} \\ &\lesssim l_i^2\norm{u - \bar{u}_i}^2_{L^{\infty}(\Omega_i)}h^{d-2}p^{4-d} \lesssim l_i^2h^{-d}p^{2d}\norm{u - \bar{u}_i}_{L^2(\Omega_i)}^2 h^{d-2}p^{4-d} \\ &\lesssim  l_i^2H_i^2h^{-d}p^{2d} h^{d-2} p^{4-d} |u|_{H^1(\Omega_i)}^2 \lesssim l_i^2\frac{H_i^2}{h^2}p^{4+d} |u|_{H^1(\Omega)}^2,
\end{align*}
where we have used triangle inequality and Cauchy-Schwarz inequality, inverse inequalities for hp-finite elements (see theorem 3.92 and theorem 4.76 in~\cite{Schwab1998} for $d=1,2$, the proof for $d=3$ is analogous), the fact that by a scaling argument plus inverse inequalities we have $|\phi_j^i|^2_{H^1({\Omega_i})} \lesssim h^{-2}p^4\norm{\phi_j^i}^2_{L^2(\Omega_i)} \lesssim h^{d-2}p^{4-d}$ (see lemma \ref{app:shape_fun_p_dep} for the $p$-dependence) and Poincaré inequality. It follows that 
\begin{equation}\label{lemma:poincare_estimate}
    \sum_{i=1}^J |I_h^{V_i} (u - \bar{u}_i)|^2_{H^1(\Omega_i)} \lesssim L^2\frac{H^2}{h^2}p^{4+d} |u|_{H^1(\Omega)}^2,
\end{equation}
where $L = \max_{i=1,\dots,J} l_i$.
We still need to estimate the term $\Pi_0 u := \sum_{i=1}^J \bar{u}_i I_h^{V_i}1$. To this end, we note that 
\begin{equation} \label{lemma:gradient_property}
    \nabla(\Pi_0 u) = \sum_{i=1}^J (\bar{u}_i - c) \nabla(I_h^{V_i}1)
\end{equation}
for any constant $c$. 
Let $N(x) := \{i : I_h^{V_i}1(x) \neq 0 \}$, it is easy to note that $|N(x)| \leq (p+1)^d$ since that is the maximum number of support points on a mesh element. For any $x \in \Omega$, we have that only indices in $N(x)$ contribute to the sum in $\Pi_0 u$. Let us pick any reference $i_0(x) \in N(x)$. If we use (\ref{lemma:gradient_property}) it holds that 
\begin{equation*}
    \nabla (\Pi_0 u)(x) = \sum_{i \in N(x)} (\bar{u}_i - \bar{u}_{i_0}) \nabla(I_h^{V_i}1)(x).
\end{equation*}
Let $\norm{\cdot}_{l^2}$ denote the classical Euclidean norm on a real-valued vector space. By the Cauchy-Schwarz inequality, we have
\begin{equation*}
    \norm{\nabla (\Pi_0 u)(x)}_{l^2}^2 \lesssim p^d\sum_{i \in N(x)} |\bar{u}_i - \bar{u}_{i_0}|^2 \norm{\nabla(I_h^{V_i}1)(x)}_{l^2}^2.
\end{equation*}
We know by definition that $\nabla(I_h^{V_i}1)$ is supported only on the boundary layer $\Omega_{i,h}$, which has thickness $O(h)$ and a uniformly bounded number of fine-cell layers. Consider the interface $\Gamma_{ij}$ between $\Omega_i$ and $\Omega_j$ and let us define $\Omega_{i,h}^j := \Omega_{i,h} \cap \Omega_j$. It is straightforward to prove that $|\Omega_{i,h}^j| \sim h|\Gamma_{ij}| \sim hH^{d-1}$ and that $\norm{\nabla(I_h^{V_i}1)(x)}_{l^2} \sim h^{-1}p^2$ where we have used hp-finite elements inverse inequalities. Hence, 
\begin{equation*}
    \int_{\Omega_{i,h}^j} \norm{\nabla(I_h^{V_i}1)(x)}_{l^2}^2 dx \lesssim p^4H^{d-1}h^{-1}.
\end{equation*} 
Now, each pair $(i,i_0)$ that shares at least an element contributes to the sum exactly once. Then, by lemma~\ref{lemma:trace_inequality} we have that 
\begin{align*}
    |\Pi_0 u|^2_{H^1(\Omega)} &\lesssim p^d \sum_{(i,i_0)} |\bar{u}_i - \bar{u}_{i_0}|^2 \int_{\Omega_{i,h}^{i_0}} \norm{\nabla(I_h^{V_i}1)(x)}_{l^2}^2 dx \\ &\lesssim p^d \sum_{(i,i_0)} \frac{H_i|u|^2_{H^1(\Omega_i)} + H_{i_0}|u|^2_{H^1(\Omega_{i_0})}}{H^{d-1}}h^{-1}p^4H^{d-1} \lesssim \frac{H}{h}p^{d+4}|u|^2_{H^1(\Omega)}.
\end{align*}
If we combine the last inequality with (\ref{lemma:poincare_estimate}) we get the thesis of this lemma.
\end{proof}
 
It follows that $C_2 = Cp^{d+4}\frac{H}{h}\Bigl(1 + L^2\frac{H}{h} \Bigr)$ where $C$ does not depend on $h, H$ or $p$. We finally prove the following lemma:
\begin{lemma}\label{lemma:scaling}
    For any $v \in V_i$ it holds that
    \begin{equation*}
        \norm{v}^2_{D_i} \lesssim \frac{p^{4+d}}{h^2}\norm{v}^2_{L^2(\Omega_i)}.
    \end{equation*}
\end{lemma}
\begin{proof}
    For any $v \in V_i$  such that $v = \sum_{j=1}^{l_i}v_j^i\phi_j^i$ where $\tilde{v} = (v_1^i, \dots, v_{l_i}^i)^T$, we have that 
\begin{align*}
    \norm{v}^2_{D_i} = (D_iv, v) &= \sum_{j=1}^{l_i}  {v_j^i}^2 |\phi_j^i|_{H^1(\Omega_i)}^2 \lesssim h^{-2}p^4\sum_{j=1}^{l_i}  {v_j^i}^2 \norm{\phi_j^i}_{L^2(\Omega_i)}^2 \lesssim h^{d-2}p^{4-d} \norm{\tilde{v}}_{l^2}^2 \\ &\lesssim h^{-2}p^{4+d}\norm{v}_{L^2(\Omega_i)}^2,
\end{align*}
where we have used inverse inequalities of hp-finite element analysis (see, for instance, lemma B.27 and B.28 in~\cite{toselli::DD}) and the fact that 
\begin{equation}\label{lemma_scaling_inequality}
    \norm{\tilde{v}}_{l^2}^2 \lesssim h^{-d}p^{2d}\norm{v}_{L^2(\Omega_i)}^2.
\end{equation} 
Indeed, let $M_i$ be the local mass matrix $(M_i)_{jk} = \int_{\Omega_i} \phi_j^i(x) \phi_k^i(x) dx$. We know that there is spectral equivalence between the mass matrix $M$ and the GLL-lumped diagonal mass matrix $\widetilde{M}$ (see e.g.~\cite{Canuto2006}). The equivalence extends to the local mass matrices $M_i$ and $\widetilde{M}_i$ as well. Hence, we have that
\begin{equation*}
    \norm{v}_{L^2(\Omega_i)}^2 = \tilde{v}^T M_i \tilde{v} \gtrsim \lambda_{min}(\widetilde{M}_i) \norm{\tilde{v}}_{l^2}^2 \gtrsim h^dp^{-2d}\norm{\tilde{v}}_{l^2}^2,
\end{equation*}
where we used the fact that $\widetilde{M}_i$ is diagonal and that $\norm{\phi_j^i}_{L^2(\Omega_i)}^2 \gtrsim h^dp^{-2d}$ from lemma~\ref{app:shape_fun_p_dep}.
\end{proof}

\subsection{The Nicolaides coarse space in a multigrid setting}\label{nicolaides_coarse_space}
To estimate $\mu_c$, let us define $Q_i^{\circ}$ as the local $L^2$ orthogonal projection on $V_i^{\circ c}$ where
\begin{equation*}
    V_i^{\circ c} := \{v_i^c|_{\Omega_i^{\circ}} : v_i^c \in V_i^c \}. 
\end{equation*} 
Consequently, let us prove the following theorem: 
\begin{theorem}\label{theo:nicolaides}
    Under the assumptions previously stated, if the fine-element geometry maps from the reference element to the physical space elements are multilinear and $p' = 0$, it holds that 
    \begin{equation*}
        \norm{v_i - Q_i v_i}^2_{D_i} \lesssim p^{4+d}\frac{H_i^2}{h^2}\Bigl(2 + l_i^2p^{d}\Bigr) |v_i|_{H^1(\Omega_i)}^2,
    \end{equation*}
    i.e.
    \begin{equation*}
        \mu_i^{-1}(V_i^c) \lesssim p^{4+d}\frac{H_i^2}{h^2}\Bigl(2 + l_i^2p^{d} \Bigr).
    \end{equation*}
\end{theorem}
\begin{proof}
From lemma~\ref{lemma:scaling}, and the definition of $Q_i$, we have that 
    \begin{equation*}
    \norm{v_i - Q_iv_i}_{D_i}^2 \leq  \norm{v_i - I_h^{V_i}(Q_i^{\circ}v_i)}_{D_i}^2 \lesssim h^{-2}p^{4+d} \norm{v_i - I_h^{V_i}(Q_i^{\circ}v_i)}_{L^2(\Omega_i)}^2.
    \end{equation*}
Now we split the term $\norm{v_i - I_h^{V_i}(Q_i^{\circ}v_i)}_{L^2(\Omega_i)}^2$:
\begin{equation*}
    \norm{v_i - I_h^{V_i}(Q_i^{\circ}v_i)}_{L^2(\Omega_i)}^2 = \norm{v_i - I_h^{V_i}(Q_i^{\circ}v_i)}_{L^2(\Omega_i^{\circ})}^2 + \norm{v_i - I_h^{V_i}(Q_i^{\circ}v_i)}_{L^2(\Omega_{i,h})}^2.
\end{equation*}
From the fact that $Q_i^{\circ}$ reproduces constant functions in $\Omega_i^{\circ}$ we get from the Poincaré inequality that
\begin{equation*}
    \norm{v_i - I_h^{V_i}(Q_i^{\circ}v_i)}_{L^2(\Omega_i^{\circ})}^2 \lesssim H_i^2 |v_i|^2_{H^1(\Omega_i^{\circ})} \lesssim H_i^2 |v_i|^2_{H^1(\Omega_i)}.
\end{equation*}
For the second term, we split again
\begin{equation*}
    \norm{v_i - I_h^{V_i}(Q_i^{\circ}v_i)}_{L^2(\Omega_{i,h})}^2 \leq \norm{v_i}_{L^2(\Omega_{i,h})}^2 + \norm{I_h^{V_i}(Q_i^{\circ}v_i)}_{L^2(\Omega_{i,h})}^2 .
\end{equation*}
From Friedrichs' inequality (see, for example,~\cite[Corollary A.15]{toselli::DD}), we have that 
\begin{equation*}
    \norm{v_i}_{L^2(\Omega_{i,h})}^2 \leq \norm{v_i}_{L^2(\Omega_{i})}^2 \lesssim H_i^2|v_i|_{H^1(\Omega_{i})}^2,
\end{equation*}
since $v_i = 0$ on $\partial\Omega_i$ or at least on a subset of $\partial\Omega_i$ with positive measure when $\Omega_i \cap \partial \Omega \neq \emptyset$. This is guaranteed by the admissible hierarchy assumptions on the overlap.
For the second term, we notice that 
    \begin{align*}
        \norm{ I_h^{V_i}(Q_i^{\circ}v_i)}_{L^2(\Omega_{i,h})}^2 &\leq \norm{ I_h^{V_i}(Q_i^{\circ}v_i)}_{L^2(\Omega_{i})}^2 \lesssim l_i^2\norm{Q_i^{\circ}v_i}_{L^{\infty}(\Omega_i^{\circ})}^2h^dp^{-d} \\ &\lesssim l_i^2h^{-d}p^{2d}\norm{Q_i^{\circ}v_i}_{L^2(\Omega_i^{\circ})}^2 h^dp^{-d} \lesssim l_i^2p^{d}\norm{v_i}_{L^2({\Omega_i})}^2 \\ &\lesssim l_i^2p^{d}H_i^2|v_i|_{H^1(\Omega_i)}^2,
    \end{align*}
    where we have used Friedrichs' inequality, hp-FEM inverse estimates, and the fact that $Q_i^{\circ}$ is an $L^2$ orthogonal projection. By combining the last four inequalities, we obtain the thesis.

\end{proof}

From theorem~\ref{xu:core_theo} and theorem~\ref{theo:nicolaides} it follows that
\begin{equation*}
    \norm{E}^2_A \leq 1 -  \frac{C}{p^{2d+8}\frac{H^3}{h^3}(2 + L^2p^d)(1 + L^2\frac{H}{h})c^D},
\end{equation*}
where $C$ does not depend on $p, h$ or $H$. In fact we recall that $L = \max_{i=1,\dots,J} l_i$. The $R^*$-tree properties imply that
\begin{itemize}
    \item if we use the cell-based agglomeration strategy then, for every $i = 1, \dots, J$, it holds that $l_i \lesssim p^dM$, i.e. $L \lesssim p^dM$
    \item if we use the point-based agglomeration strategy then, for every $i = 1, \dots, J$, it holds that $l_i \lesssim M$, i.e. $L \lesssim M$.
\end{itemize}

Even though this estimate is not sharp, we can still observe that if the ratio $H/h$ is constant and $p$ is fixed, then the multigrid algorithm converges uniformly with respect to the mesh size. We remark that if only the ratio $H/h$ is fixed the bound deteriorates as $p$ increases and uniform convergence cannot hold.
 
\subsection{Generalization to higher order polynomials}\label{high_order_generalization}

Now, let $p \geq 1$ and $p' \geq 1$. Since (\ref{eq:C1}) and (\ref{eq:C2}) still hold in the general setting, we only need to estimate $\mu_c$. To this end, we prove the following:
\begin{theorem} \label{theo:core}
Under the assumptions previously stated, if the fine-element geometry maps from the reference element to the physical space elements are multilinear and $p' \geq 1$ then,
    \begin{equation*}
        \norm{v_i - Q_i v_i}^2_{D_i} \lesssim p^{4+d}\frac{H_i^2}{h^2}\Bigl(1 + l_i^2p^{d} + {p'}^{-2} \Bigr) |v_i|_{H^1(\Omega_i)}^2,
    \end{equation*}
    i.e.
    \begin{equation*}
        \mu_i^{-1}(V_i^c) \lesssim p^{4+d}\frac{H_i^2}{h^2}\Bigl(1 + l_i^2p^{d} + {p'}^{-2} \Bigr).
    \end{equation*}
\end{theorem}
\begin{proof}
    From lemma~\ref{lemma:scaling}, and the definition of $Q_i$, we have that 
    \begin{equation*}
    \norm{v_i - Q_iv_i}_{D_i}^2 \leq  \norm{v_i - I_h^{V_i}(Q_i^{\circ}v_i)}_{D_i}^2 \lesssim h^{-2}p^{4+d} \norm{v_i - I_h^{V_i}(Q_i^{\circ}v_i)}_{L^2(\Omega_i)}^2.
    \end{equation*}
    Now, if we split the term $\norm{v_i - I_h^{V_i}(Q_i^{\circ}v_i)}_{L^2(\Omega_i)}^2$ as we did in the proof of theorem~\ref{theo:nicolaides}, we obtain the following
    \begin{equation*}
        \norm{v_i -  I_h^{V_i}(Q_i^{\circ}v_i)}_{L^2(\Omega_i^{\circ})}^2 \lesssim \frac{H_i^2}{{p'}^2} |v_i|^2_{H^1(\Omega_i^{\circ})} \lesssim \frac{H_i^2}{{p'}^2} |v_i|^2_{H^1(\Omega_i)},
    \end{equation*}
    where we have used the Bramble-Hilbert lemma, classical Aubin-Nitsche duality argument, and estimates from hp-finite element analysis (see theorem 4.1 in~\cite{hp::review}, but also~\cite{Schwab1998,BABUSKA19905,hp:quasi_uniform}). Indeed, $ I_h^{V_i}(Q_i^{\circ})$ preserves polynomials up to total degree $p'$ in $\Omega_i^{\circ}$ since the space of polynomials on $\Omega_i^{\circ}$ of total degree $p'$ is a subset of $V_i^{\circ c}$.
    If we also split the term in the overlap $\Omega_{i,h}$ like we did in theorem~\ref{theo:nicolaides}, we obtain 
   \begin{equation*}
    \norm{v_i}_{L^2(\Omega_{i,h})}^2 \lesssim H_i^2|v_i|_{H^1(\Omega_{i})}^2 .
    \end{equation*}
    By repeating the same arguments as in the proof of theorem~\ref{theo:nicolaides}, we also have that 
    \begin{equation*}
        \norm{ I_h^{V_i}(Q_i^{\circ}v_i)}_{L^2(\Omega_{i,h})}^2 \lesssim l_i^2p^{d}H_i^2|v_i|_{H^1(\Omega_i)}^2,
    \end{equation*}
     By combining the last three inequalities, we obtain the thesis.
\end{proof}
From theorem~\ref{xu:core_theo} and theorem~\ref{theo:core} it follows that
\begin{equation*}
    \norm{E}^2_A \leq 1 -  \frac{C}{p^{2d+8}\frac{H^3}{h^3}(1 + {p'}^{-2} + L^2p^d)(1 + L^2\frac{H}{h})c^D},
\end{equation*}
where $C$ does not depend on $p,p',h$ and $H$. The estimates on $L$ are the same as in the previous section.
Again our estimate is not sharp but we still observe that if the ratio $H/h$ is constant and $p$ is fixed, then the multigrid algorithm converges uniformly with respect to the mesh size. Even though the estimate deteriorates as $p$ increases and uniform convergence cannot hold unless $p$ is fixed, we remark that the estimate improves as $p'$ increases.

\subsection{The simplicial elements case}
In the simplicial case (i.e. $\mathcal{R}^p = \mathcal{P}^p$) lemma~\ref{app:shape_fun_p_dep} does not hold. It is well known that tensor-product GLL nodes are not available for simplices. Uniformly spaced nodes are ill-conditioned~\cite{BOS198343}, a possible solution would be to employ Fekete points~\cite{Fekete1923, Fekete2000} on the reference simplex as nodes for the Lagrange basis functions. In fact, in~\cite{Fekete2000} it is shown that, given a reference simplex $\hat{T}$, the Lagrange basis functions $\phi_{\mathbf{j}}$ associated with Fekete points satisfy
\begin{equation*}
    \norm{\phi_{\mathbf{j}}}^2_{L^{\infty}(\hat{T})} \leq 1,
\end{equation*}
i.e.
\begin{equation*}
    \norm{\phi_{\mathbf{j}}}^2_{L^{2}(\hat{T})} \lesssim 1,
\end{equation*}
since we are considering a reference element. The main issue is that when we use Fekete points the $p$-dependence in the proof of lemma~\ref{lemma:scaling} is not necessarily true anymore: only the $h$-dependence still holds.
If we adapt the proof of lemma~\ref{lemma:stable_decomposition} and theorem~\ref{theo:core}, by changing the estimate in lemma~\ref{app:shape_fun_p_dep} with the estimate for simplicial elements, we can only state that the two-level error propagation operator satisfies the following estimate:
\begin{equation*}
    \norm{E}^2_A \leq 1 -  \frac{C(p)}{p^{d+4}\frac{H^3}{h^3}(1 + {p'}^{-2} + L^2p^{2d})(1 + L^2p^{d}\frac{H}{h})c^D},
\end{equation*} 
where $C(p)$ does not depend on $H,h,p'$ but it does depend on $p$ and the dependence is tied to the $p$-dependence of the constant in lemma~\ref{lemma:scaling}. This estimate is strictly worse than the one we obtained for hypercubes. Nevertheless, it still confirms the qualitative behavior of the algorithm: if the ratio $H/h$ is constant and $p$ is fixed, then the multigrid algorithm converges uniformly with respect to the mesh size. In any case, the bound deteriorates as $p$ increases and uniform convergence cannot hold unless $p$ is fixed. We remark that the estimate still improves as $p'$ increases.

\begin{remark}
   In our numerical experiments we use uniformly spaced nodes on the simplex, as they are the only kind of nodes available for simplices in the~\textsc{deal.II}~\cite{dealII95,dealIIdesign} library. Even though uniformly spaced nodes are ill-conditioned, we never consider a polynomial degree $p > 3$ in the experiments because of our computational-resource limitations. It is known that for $p=1,2$ uniformly spaced nodes are Fekete points, while for $p \geq 3$ Fekete points are not uniformly spaced~\cite{BOS198343}. Fekete points on simplices behave similarly to GLL nodes on hypercubes: GLL nodes coincide with Fekete points on hypercubes~\cite{Fekete_is_GLL}.
\end{remark}

\section{Numerical experiments}\label{sec:numerics}

In this section, we investigate the capabilities of the preconditioner introduced in Section~\ref{sec:preconditioner}
in different scenarios by varying the spatial dimension, geometry, and polynomial degrees. We solve the model problem defined in Section~\ref{sec:model_problem}, where $f$ is a manufactured forcing term such that the analytical solution is $u = \prod_{i=1}^d\sin(\pi x_i)$, with $d=2,3$.

As with any multilevel method, our multigrid preconditioner has several key components:
a hierarchy of meshes, intergrid transfer operators between levels, a smoother, a sequence of operators $\{A_l\}_{l=0}^{L-1}$, and a coarse-grid solver.
We choose a degree-$3$ Chebyshev-accelerated Jacobi smoother~\cite{ADAMS2003593}. On each level, it uses the precomputed inverse of the matrix diagonal and applications of the level operator to compute the residuals in the Jacobi-type iteration. Its parameters are determined from eigenvalue estimates obtained through Lanczos iterations. In our simulations, we use $2$ pre- and $2$ post-smoothing steps. For the coarse-grid solver, we use the sparse direct solver \textsc{MUMPS}~\cite{mumps}.

We compare our multilevel approach with the AMG implementation in the \textsc{Trilinos} ML package~\cite{Trilinos}. The preconditioned conjugate gradient solver applies one V-cycle of either the agglomerated multigrid or the algebraic multigrid preconditioner. We use the multilevel methods as preconditioners, since it is well known that multigrid
methods are more robust when used as preconditioners rather than solvers~\cite{HariStadlerBirosHighOrderMG}. The AMG preconditioner parameters, reported in Table~\ref{tab:amg_params}, are set according
to best practices.

\begin{table}[h]
    \centering
    \begin{tabular}{ll}
        \toprule
        Parameter             & Value                \\
        \midrule
        Smoother              & Chebyshev            \\
        Smoother degree       & 2                    \\
        Smoother sweeps       & 2                    \\
        V-cycle applications  & 1                    \\
        Aggregation type      & Smoothed aggregation \\
        Aggregation threshold & $0.01$              \\
        Max size coarse level & 2000                 \\
        Coarse solver         & Amesos-KLU           \\
        \bottomrule
    \end{tabular}
    \caption{Parameters for \textsc{TrilinosML} preconditioner~\cite{Trilinos} (Trilinos 14.4.0). When using polynomial degrees higher than one, $\mathtt{higher\_order\_elements}$ is set to true.}
    \label{tab:amg_params}
\end{table}

For all tests, we terminate the preconditioned conjugate gradient (PCG) iteration when either the absolute residual norm falls below $10^{-9}$ or the residual has been reduced by $6$ orders of magnitude.

In our experiments, we always skip the leaves level, as described in Remark~\ref{rmrk:parameters}. For each set of experiments, we report the parameter $m$ used for point-based agglomeration, $m_{points}$, and cell-based agglomeration, $m_{cells}$. For both strategies, the parameter $M$ is always equal to $2m$. To ensure fast multigrid convergence, we do not use every available R-tree level in the multilevel preconditioner. Instead, we construct the coarsening hierarchy from the bottom up, starting from the original grid and using the R-tree-based agglomeration strategy described in Section~\ref{sec:rtree}. We specify the number of levels used in each experiment.

Table~\ref{tab:legend} describes the entries used in the subsequent tables of experimental results, while Figure~\ref{fig:meshes} shows all the meshes used in the experiments. In two dimensions, we consider both structured and unstructured meshes of the unit square $[0,1]^2$. The three-dimensional tests range from a structured mesh of $[0,1]^3$ to idealized and realistic meshes of a left ventricle. We also consider two simplicial meshes: an unstructured mesh of the unit square and a liver mesh.
We specifically employ unstructured and simplicial meshes to assess the applicability of our framework in practical scenarios.

All numerical experiments are performed using the C++ software project \textsc{polyDEAL}~\cite{polydeal}, which is built on the finite element library~\textsc{deal.II}~\cite{dealII95,dealIIdesign}. It provides
building blocks for solving PDEs using high-order discontinuous Galerkin methods on polytopic meshes. In particular, it supports the R-tree-based agglomeration strategy
and the agglomerated multigrid framework described in the previous section. All our experiments are publicly available
in the \textsc{polyDEAL} GitHub repository~\cite{polydeal}, which also provides detailed instructions for running the example programs.

\begin{table}[H]
    \centering
    \begin{tabular}{c|c}
    \hline
    Entry & Meaning \\
    \hline
       Point agglo & Point-based agglomeration strategy was used \\
       Cell agglo & Cell-based agglomeration strategy was used \\
       $\mathcal{Q}^{p'}$ coarse elements & The polynomial degree $p'$ of the coarse space employed \\
       Ref.  &  Number of refinements applied to the starting mesh \\
       DoFs & Number of degrees of freedom \\
       lvls  &  Number of levels of the multilevel preconditioner \\
       $H$ & Maximum box size on the coarsest level  \\
       $h$  & Maximum element size on the finest level \\
       $H/h$ & Ratio between $H$ and $h$ \\
       Iters & Number of iterations performed by PCG \\
       Cond. & Condition number estimated by PCG\\
       \hline
    \end{tabular}
    \caption{Legend for the entries of the subsequent tables.}
    \label{tab:legend}
\end{table}

\begin{figure}[!htb]
  \centering
  \subfloat[Structured square]{%
  \label{fig:structured_square_view}%
  \includegraphics[width=0.28\textwidth]{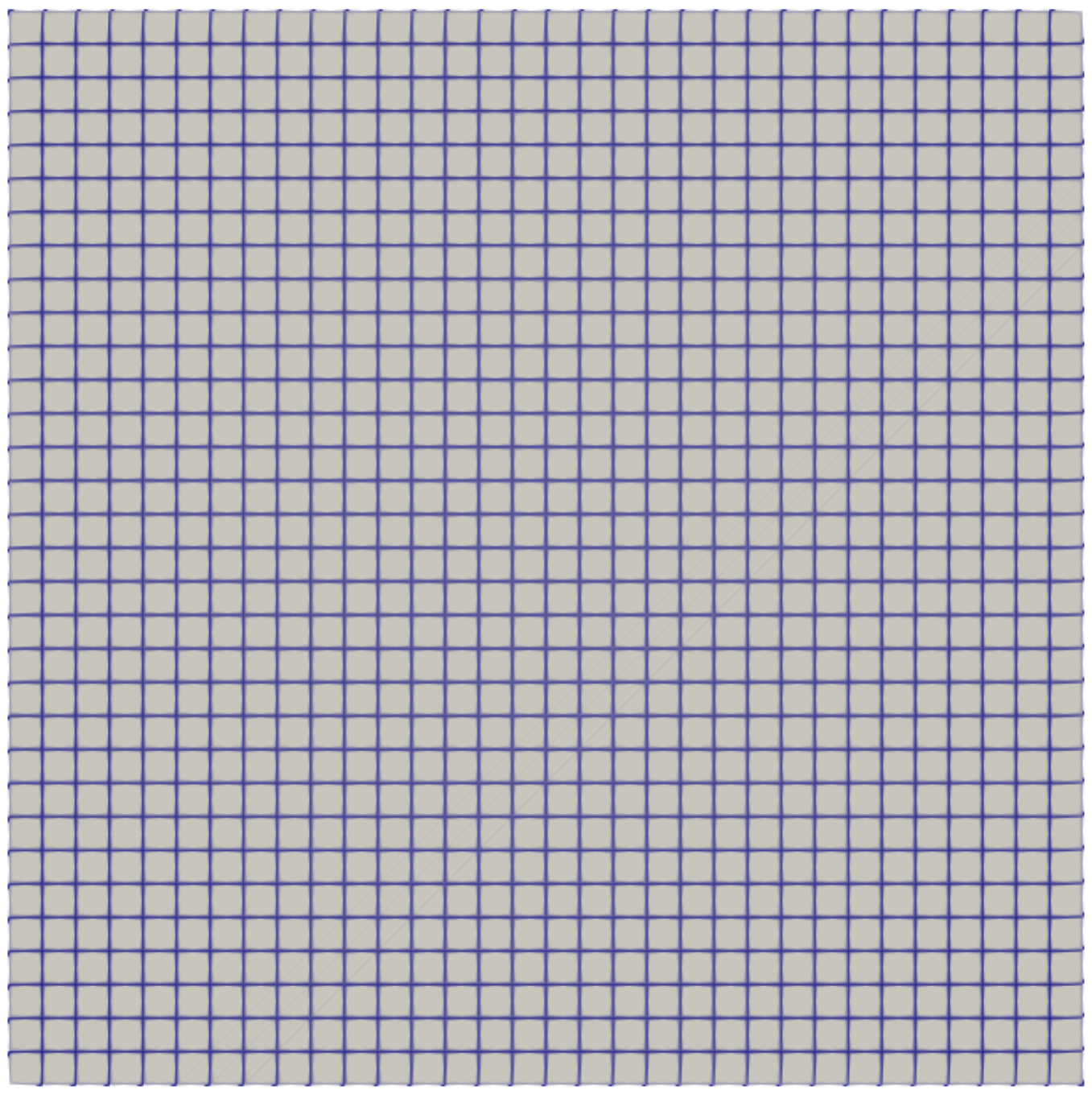}%
  }\hspace{7mm}
  \subfloat[Unstructured square]{%
  \label{fig:unstructured_square_view}%
  \includegraphics[width=0.28\textwidth]{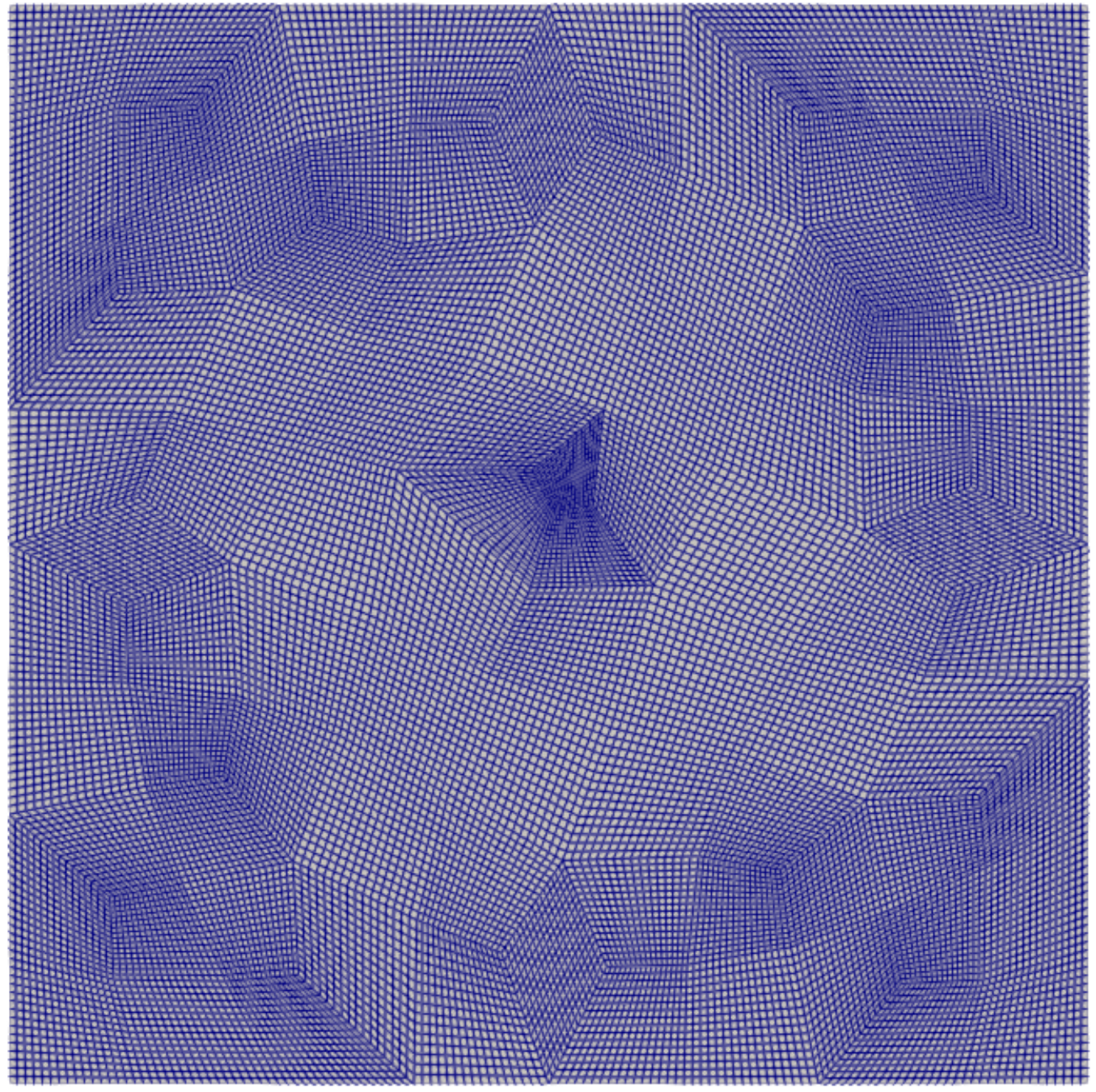}%
  }\hspace{7mm}
  \subfloat[Unstructured square (simplices)]{\label{fig:unstructured_square_simplexes_view} \includegraphics[width=0.28\textwidth]{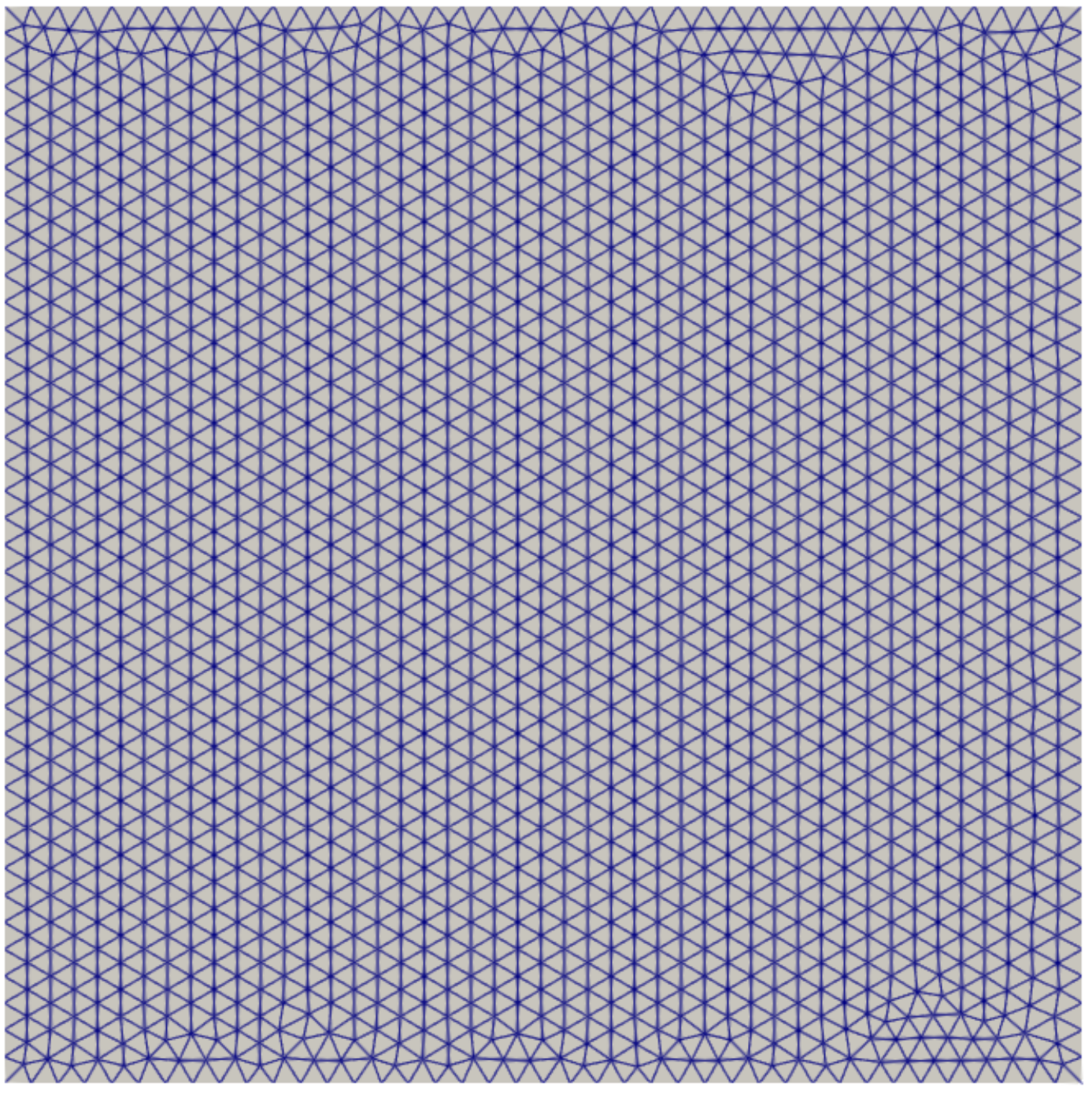}}
  
  \medskip
  \vspace{2mm}
  \subfloat[Structured cube]{%
  \label{fig:cube_view}%
  \includegraphics[width=0.28\textwidth]{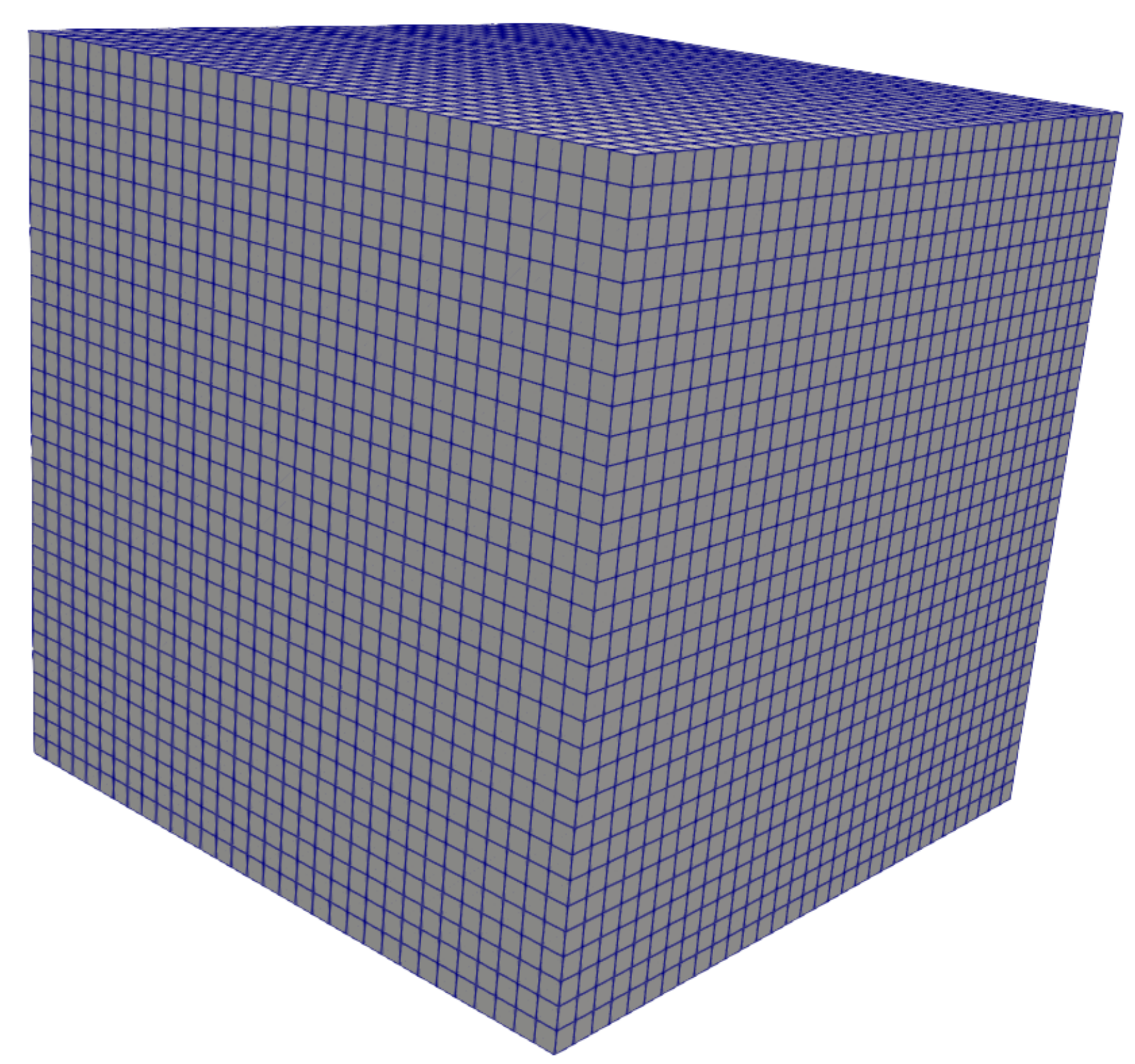}%
  }\hspace{7mm}%
  \subfloat[Idealized left ventricle mesh]{%
  \label{fig:id_lv_mesh}%
  \includegraphics[width=0.28\textwidth]{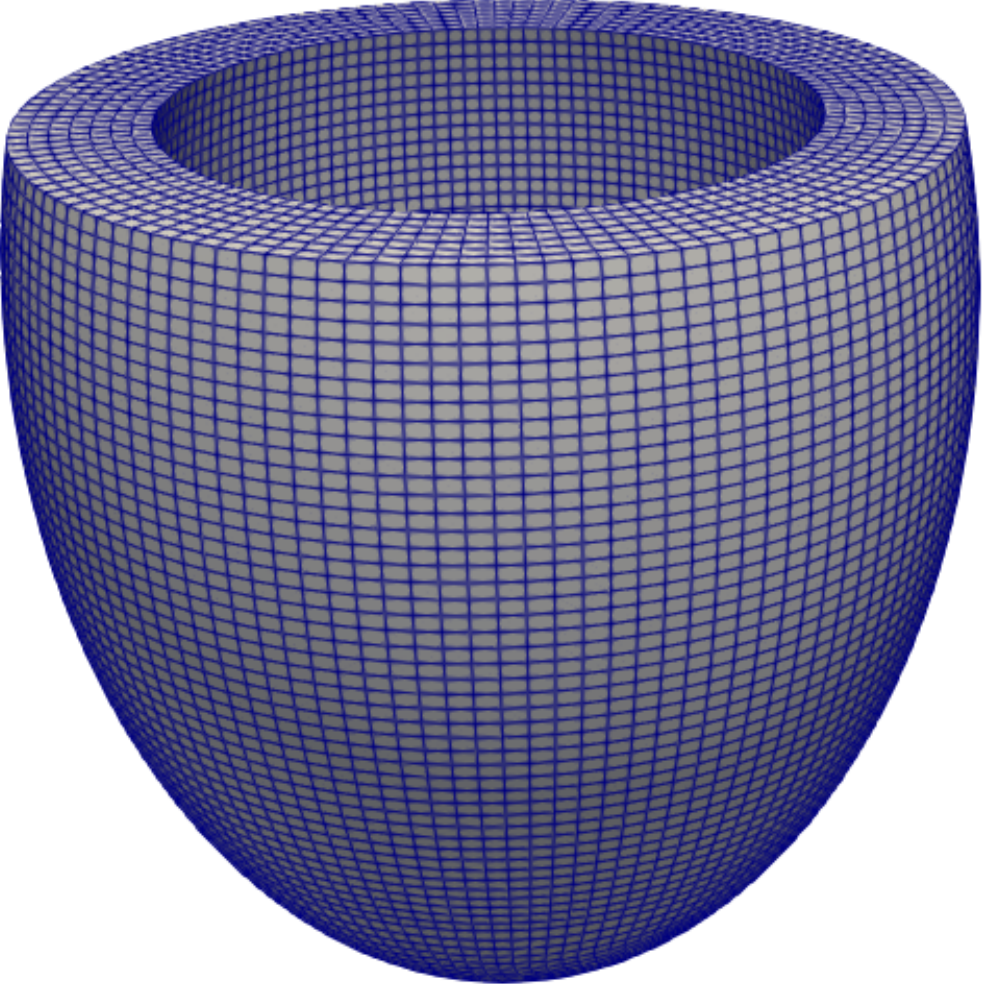}%
  }\hspace{7mm}%
  \subfloat[Realistic left ventricle mesh]{%
  \label{fig:real_lv_mesh}%
  \includegraphics[width=0.28\textwidth]{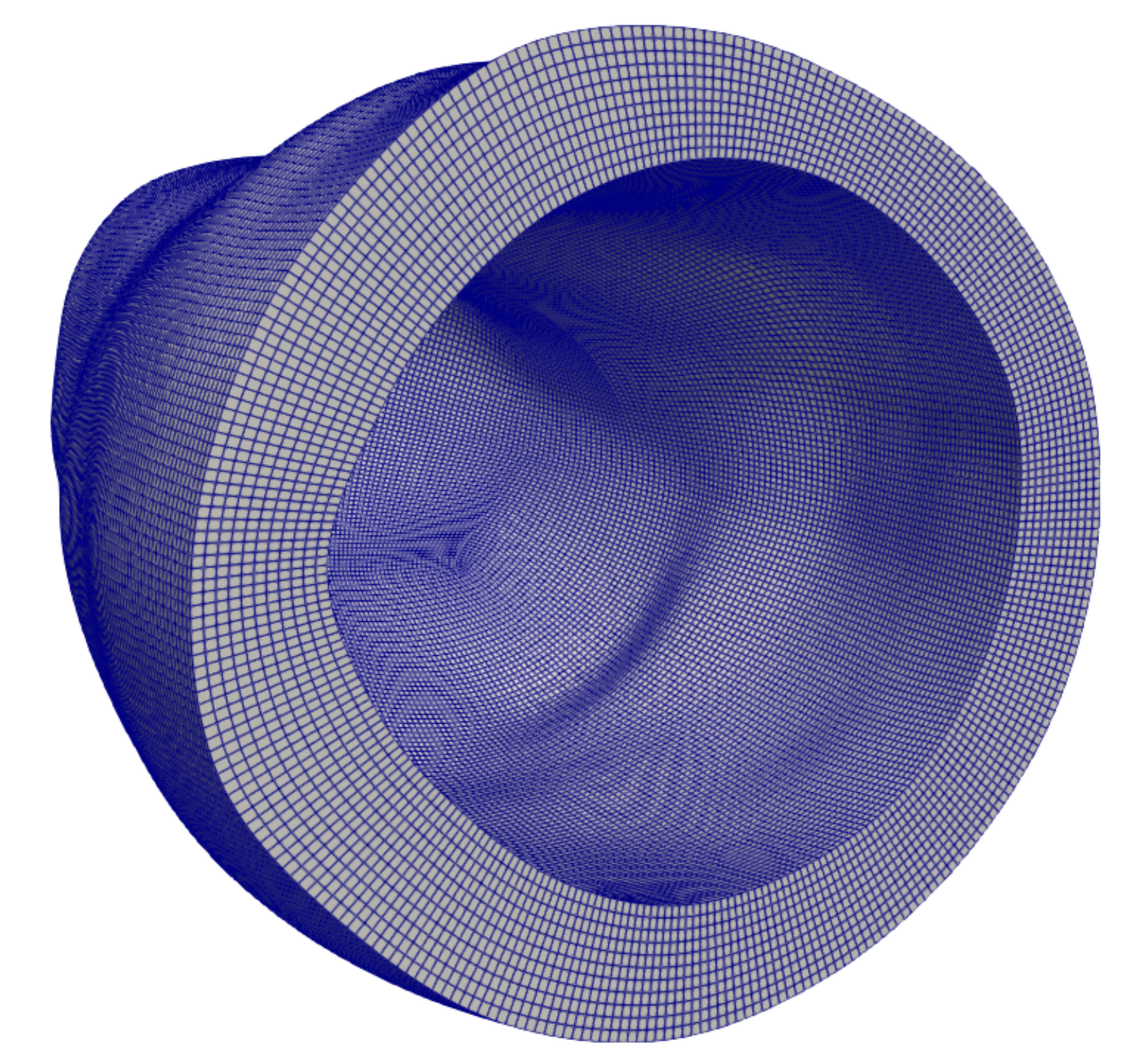}
    }
  \medskip
  \vspace{2mm}
  \subfloat[Liver mesh (simplices)]{%
  \label{fig:liver_mesh}%
  \includegraphics[width=0.28\textwidth]{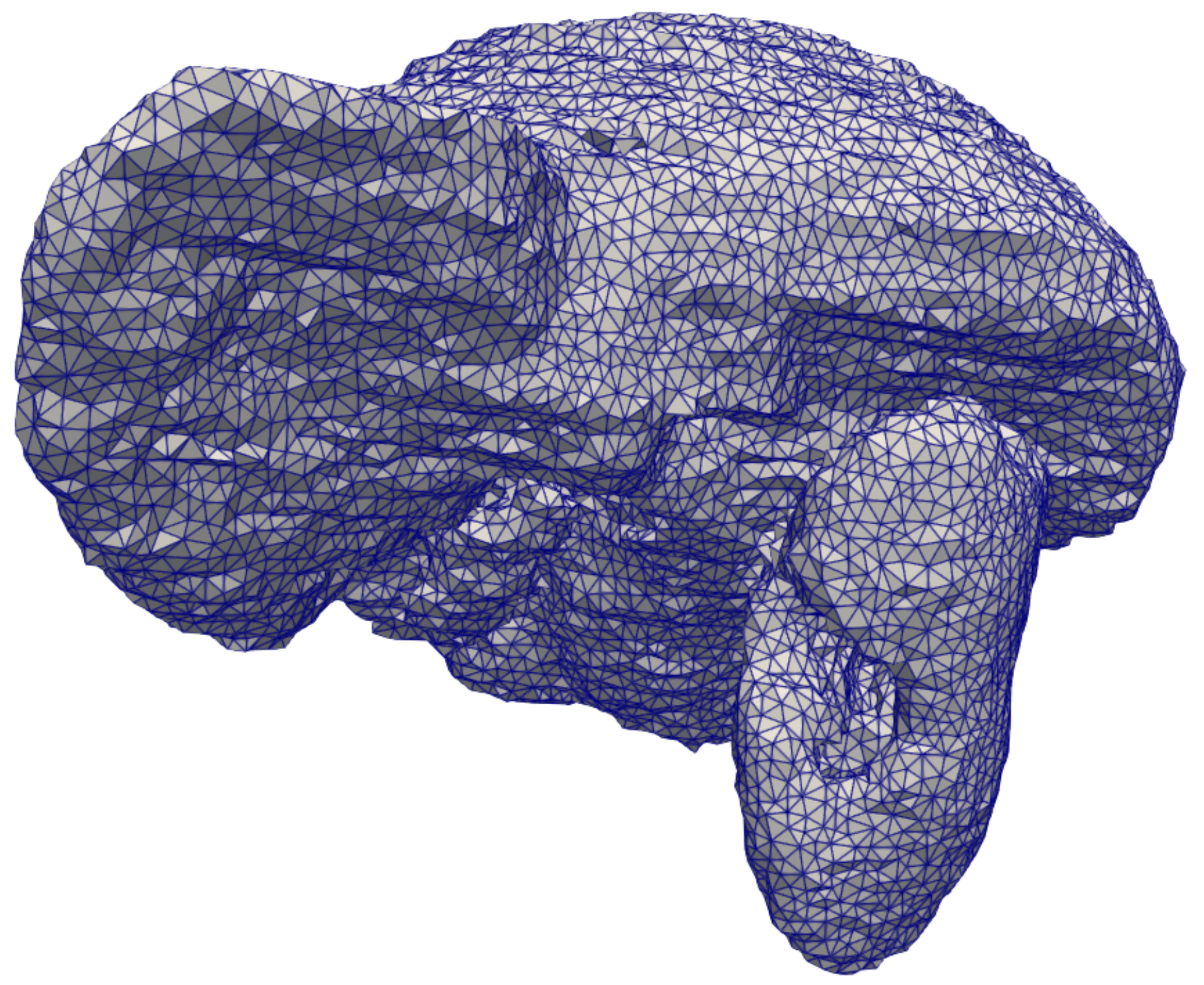}%
  }
  \caption{The set of meshes used in the numerical experiments.}\label{fig:meshes}
  \end{figure}

\subsection{Structured square}
To validate the theoretical findings, we start by considering the simplest case of a structured discretization of the unit square $[0,1]^2$. In this case, the starting mesh at zero refinements is the single cell unit square. 

In Table~\ref{table:fixed_struc_square_Q1}, we try to keep the ratio $H/h$ constant by fixing the number of levels of the multilevel preconditioner. With cell-based agglomeration, the iteration counts and condition-number estimates stabilize at nearly constant values for both choices of $p'$. As suggested by the two-level theory, the case $p'=0$ stabilizes at a higher iteration count, whereas $p'=1$ yields lower and stable iteration counts.

The point-based agglomeration strategy produces agglomerates of lower quality in a structured setting. In any case, the theory still holds, and we can see that the iteration counts behave similarly to those for cell-based agglomeration.

\begin{table}[!htb]
\centering
\begin{subtable}{\textwidth}
\centering
    \begin{tabular}{c c c c | c c c c}
        \hline
        \multicolumn{4}{c|}{Point agglo}  & \multicolumn{2}{c}{$\mathcal{Q}^0$ coarse elements} & \multicolumn{2}{c}{$\mathcal{Q}^1$ coarse elements} \\
        \hline
        Ref. & DoFs & lvls & $H/h$ & Iters & Cond. & Iters & Cond. \\
        \hline
        6 & 4,225 & 2 & 4.74 & 12 & 3.84 & 4 & 1.04 \\
        7 & 16,641 & 2 & 4.74 & 15 & 4.30 & 4 & 1.04 \\
        8 & 66,049 & 2 & 4.74 & 16 & 4.48 & 4 & 1.04 \\
        9 & 263,169 & 2 & 5.39 & 17 & 4.54 & 4 & 1.04 \\
        10 & 1,050,625 & 2 & 4.74 & 18 & 4.65 & 4 & 1.04 \\
        \hline
    \end{tabular}
\end{subtable}
\\[1em]
\begin{subtable}{\textwidth}
\centering
    \begin{tabular}{c c c c | c c c c}
        \hline
        \multicolumn{4}{c|}{Cell agglo}  & \multicolumn{2}{c}{$\mathcal{Q}^0$ coarse elements} & \multicolumn{2}{c}{$\mathcal{Q}^1$ coarse elements} \\
        \hline
        Ref. & DoFs & lvls & $H/h$ & Iters & Cond. & Iters & Cond. \\
        \hline
        6 & 4,225 & 2 & 4.00 & 11 & 3.63 & 4 & 1.03 \\
        7 & 16,641 & 2 & 4.00 & 14 & 3.98 & 4 & 1.02 \\
        8 & 66,049 & 2 & 4.00 & 15 & 4.08 & 3 & 1.01 \\
        9 & 263,169 & 2 & 4.00 & 15 & 4.10 & 3 & 1.01 \\
        10 & 1,050,625 & 2 & 4.00 & 15 & 4.11 & 3 & 1.01 \\
        \hline
    \end{tabular}
\end{subtable}

\caption{Experimental results trying to keep the ratio $H/h$ fixed between refinements with $\mathcal{Q}^1$ elements on the structured square mesh. $m_{cells} = 2$ and $m_{points} = 2$.}
\label{table:fixed_struc_square_Q1}
\end{table}

In Table~\ref{table:ref_struc_square_Q1}, we show the iteration counts under $h$-refinement. As the ratio $H/h$ increases, so do the iteration counts and condition numbers. The increase in iterations is higher for the point-based agglomeration as expected given the lower quality agglomerates. Even though we have not developed a multilevel theory, it is reasonable to expect such behavior given the two-level theory. By refining without controlling the coarsest level, the ratio $H/h$ grows unboundedly, resulting in a deterioration of the coarse space. 
In Table~\ref{table:ref_struc_square_Q1_trilinos}, we report the iteration counts for Trilinos AMG. Our multilevel preconditioner is competitive with Trilinos AMG in terms of iteration counts.

\begin{table}[!htb]
\centering
\begin{subtable}{\textwidth}
\centering
    \begin{tabular}{c c c c | c c c c}
        \hline
        \multicolumn{4}{c|}{Point agglo}  & \multicolumn{2}{c}{$\mathcal{Q}^0$ coarse elements} & \multicolumn{2}{c}{$\mathcal{Q}^1$ coarse elements} \\
        \hline
        Ref. & DoFs & lvls & $H/h$ & Iters & Cond. & Iters & Cond. \\
        \hline
        6 & 4,225 & 3 & 8.63 & 12 & 5.11 & 5 & 1.07 \\
        7 & 16,641 & 4 & 18.38 & 17 & 9.70 & 5 & 1.19 \\
        8 & 66,049 & 5 & 33.59 & 24 & 18.52 & 7 & 1.54 \\
        9 & 263,169 & 6 & 64.66 & 35 & 36.37 & 10 & 2.16 \\
        10 & 1,050,625 & 7 & 128.50 & 49 & 71.66 & 14 & 3.73 \\
        \hline
    \end{tabular}
\end{subtable}
\\[1em]
\begin{subtable}{\textwidth}
\centering
    \begin{tabular}{c c c c | c c c c}
        \hline
        \multicolumn{4}{c|}{Cell agglo}  & \multicolumn{2}{c}{$\mathcal{Q}^0$ coarse elements} & \multicolumn{2}{c}{$\mathcal{Q}^1$ coarse elements} \\
        \hline
        Ref. & DoFs & lvls & $H/h$ & Iters & Cond. & Iters & Cond. \\
        \hline
        6 & 4,225 & 2 & 4.00 & 11 & 3.63 & 4 & 1.03 \\
        7 & 16,641 & 3 & 8.00 & 15 & 6.85 & 4 & 1.02 \\
        8 & 66,049 & 4 & 16.00 & 22 & 13.25 & 4 & 1.04 \\
        9 & 263,169 & 5 & 32.00 & 30 & 26.13 & 6 & 1.35 \\
        10 & 1,050,625 & 6 & 64.00 & 42 & 52.10 & 8 & 2.05 \\
        \hline
    \end{tabular}
\end{subtable}
\caption{Experimental results increasing refinements with $\mathcal{Q}^1$ elements on the structured square mesh. $m_{cells} = 2$ and $m_{points} = 2$.}
\label{table:ref_struc_square_Q1}
\end{table}

In Table~\ref{table:fixed_struc_square_Q3}, we use $\mathcal{Q}^3$ conforming elements to discretize the model problem. We use $p'=1$ or $p'=2$ as coarse space degree. For the higher degree coarse space we need to increase $m_{points}$, otherwise the agglomerates might not have enough points to build the transfer operators (see remark~\ref{rmrk:parameters}). Again, the theory for the two-level preconditioner is confirmed since the iteration counts are stable. For the point-based agglomeration strategy, increasing $p'$ does not provide substantial gains. In fact, even though we increased $p'$ we also increased $H/h$. Some improvements in terms of iteration counts can be seen for the  cell-based agglomeration. 

\begin{table}[!htb]
\centering
\begin{subtable}{\textwidth}
\centering
    \begin{tabular}{c c | c c c c | c c c c}
        \hline
        \multicolumn{2}{c|}{Point agglo}  & \multicolumn{4}{c|}{$\mathcal{Q}^1$ coarse elements} & \multicolumn{4}{c}{$\mathcal{Q}^2$ coarse elements} \\
        \hline
        Ref. & DoFs & lvls & $H/h$ & Iters & Cond. & lvls & $H/h$ & Iters & Cond. \\
        \hline
        5 & 9,409 & 2 & 2.18 & 6 & 1.40 & 2 & 3.26 & 5 & 1.24 \\
        6 & 37,249 & 2 & 2.18 & 6 & 1.40 & 2 & 3.68 & 6 & 1.44 \\
        7 & 148,225 & 2 & 2.37 & 6 & 1.46 & 2 & 5.00 & 6 & 1.57 \\
        8 & 591,361 & 2 & 2.18 & 6 & 1.42 & 2 & 3.34 & 5 & 1.26 \\
        9 & 2,362,369 & 2 & 2.37 & 6 & 1.39 & 2 & 3.74 & 5 & 1.34 \\
        \hline
    \end{tabular}
\end{subtable}\\[1em]
\begin{subtable}{\textwidth}
\centering
    \begin{tabular}{c c c c | c c c c}
        \hline
        \multicolumn{4}{c|}{Cell agglo}  & \multicolumn{2}{c}{$\mathcal{Q}^1$ coarse elements} & \multicolumn{2}{c}{$\mathcal{Q}^2$ coarse elements} \\
        \hline
        Ref. & DoFs & lvls & $H/h$ & Iters & Cond. & Iters & Cond. \\
        \hline
        5 & 9,409 & 2 & 4 & 9 & 2.68 & 5 & 1.26 \\
        6 & 37,249 & 2 & 4 & 9 & 2.70 & 5 & 1.26 \\
        7 & 148,225 & 2 & 4 & 8 & 2.66 & 5 & 1.26 \\
        8 & 591,361 & 2 & 4 & 8 & 2.66 & 5 & 1.26 \\
        9 & 2,362,369 & 2 & 4 & 8 & 2.63 & 5 & 1.26 \\
        \hline
    \end{tabular}
\end{subtable}

\caption{Experimental results trying to keep the ratio $H/h$ fixed between refinements with $\mathcal{Q}^3$ elements on the structured square mesh. When the coarse space has degree $p' = 1$ we set $m_{cells} = 2$ and $m_{points} = 2$. When the coarse space has degree $p' = 2$ we set $m_{cells} = 2$ and $m_{points} = 4$.}
\label{table:fixed_struc_square_Q3}
\end{table}

In Table~\ref{table:ref_struc_square_Q3}, we use conforming $\mathcal{Q}^3$ elements to discretize the model problem and report iteration counts under $h$-refinement. We observe behavior similar to that obtained with $\mathcal{Q}^1$ elements. Cell-based agglomeration performs slightly better than point-based agglomeration. Table~\ref{table:ref_struc_square_Q1_trilinos} also shows that our multilevel preconditioner requires fewer iterations than Trilinos AMG. Indeed, AMG is known to perform poorly on some high-order discretizations~\cite{amg:high_order}.

\begin{table}[!htb]
\centering
\begin{subtable}{\textwidth}
\centering
    \begin{tabular}{c c | c c c c | c c c c}
        \hline
        \multicolumn{2}{c|}{Point agglo}  & \multicolumn{4}{c|}{$\mathcal{Q}^1$ coarse elements} & \multicolumn{4}{c}{$\mathcal{Q}^2$ coarse elements} \\
        \hline
        Ref. & DoFs & lvls & $H/h$ & Iters & Cond. & lvls & $H/h$ & Iters & Cond. \\
        \hline
        5 & 9,409 & 3 & 4.17 & 8 & 1.88 & 2 & 3.26 & 5 & 1.24 \\
        6 & 37,249 & 4 & 7.82 & 10 & 2.66 & 3 & 8.63 & 7 & 1.49 \\
        7 & 148,225 & 5 & 16.83 & 14 & 4.34 & 3 & 11.12 & 8 & 2.03 \\
        8 & 591,361 & 6 & 31.15 & 20 & 7.47 & 4 & 24.96 & 11 & 2.96 \\
        9 & 2,362,369 & 7 & 59.62 & 28 & 13.52 & 5 & 63.95 & 15 & 4.45 \\
        \hline
    \end{tabular}
\end{subtable}\\[1em]
\begin{subtable}{\textwidth}
\centering
    \begin{tabular}{c c c c | c c c c}
        \hline
        \multicolumn{4}{c|}{Cell agglo}  & \multicolumn{2}{c}{$\mathcal{Q}^1$ coarse elements} & \multicolumn{2}{c}{$\mathcal{Q}^2$ coarse elements} \\
        \hline
        Ref. & DoFs & lvls & $H/h$ & Iters & Cond. & Iters & Cond. \\
        \hline
        5 & 9,409 & 3 & 8 & 10 & 2.71 & 6 & 1.26 \\
        6 & 37,249 & 4 & 16 & 11 & 2.75 & 6 & 1.26 \\
        7 & 148,225 & 5 & 32 & 14 & 4.30 & 6 & 1.21 \\
        8 & 591,361 & 6 & 64 & 18 & 7.53 & 7 & 1.33 \\
        9 & 2,362,369 & 7 & 128 & 25 & 14.08 & 8 & 1.74 \\
        \hline
    \end{tabular}
\end{subtable}

\caption{Experimental results increasing refinements with $\mathcal{Q}^3$ elements on the structured square mesh. When the coarse space has degree $p' = 1$ we set $m_{cells} = 2$ and $m_{points} = 2$. When the coarse space has degree $p' = 2$ we set $m_{cells} = 2$ and $m_{points} = 4$.}
\label{table:ref_struc_square_Q3}
\end{table}

\begin{table}[!htb]
    \centering
    \begin{subtable}{0.4\textwidth}
    \centering
    \begin{tabular}{c c}
        \hline
        \multicolumn{2}{c}{Trilinos AMG} \\
        \hline
        Ref. & Iters \\
        \hline
        6 & 7 \\
        7 & 7 \\
        8 & 8 \\
        9 & 9 \\
        10 & 10 \\
        \hline
    \end{tabular}
    \end{subtable}
    \begin{subtable}{0.4\textwidth}
    \centering
    \begin{tabular}{c c}
        \hline
        \multicolumn{2}{c}{Trilinos AMG} \\
        \hline
        Ref. & Iters \\
        \hline
        5 & 20 \\
        6 & 21 \\
        7 & 23 \\
        8 & 24 \\
        9 & 27 \\
        \hline
    \end{tabular}
    \end{subtable}
    \caption{Trilinos AMG iterations while increasing refinements with $\mathcal{Q}^1$ elements (first table) and $\mathcal{Q}^3$ elements (second table) on the structured square mesh.}
    \label{table:ref_struc_square_Q1_trilinos}
\end{table}

\subsection{Unstructured square}
Here, we consider an unstructured discretization of the unit square $[0,1]^2$. The initial, unrefined mesh consists of 91 quadrilateral cells. In Table~\ref{table:fixed_unstruc_square_Q1}, we try to keep the ratio $H/h$ constant by fixing the number of levels of the multilevel preconditioner. In the unstructured case, keeping this ratio constant is difficult because the cells are not axis-aligned and R-tree agglomeration is heuristic. For $p'=0$, the iteration counts and condition-number estimates increase moderately with $H/h$, whereas for $p'=1$ the iteration counts increase more slowly. This behavior occurs in both the two-level and multilevel cases. On this mesh, cell-based and point-based agglomeration perform similarly. This is expected because, unlike on the axis-aligned structured square mesh, the cell-based agglomerates do not naturally conform to the geometry of the fine mesh.

\begin{table}[!htb]
\centering
\begin{subtable}{\textwidth}
\centering
    \begin{tabular}{c c c c | c c c c}
        \hline
        \multicolumn{4}{c|}{Point agglo}  & \multicolumn{2}{c}{$\mathcal{Q}^0$ coarse elements} & \multicolumn{2}{c}{$\mathcal{Q}^1$ coarse elements} \\
        \hline
        Ref. & DoFs & lvls & $H/h$ & Iters & Cond. & Iters & Cond. \\
        \hline
        2 & 1,521 & 4 & 9.70 & 11 & 3.52 & 5 & 1.12 \\
        3 & 5,953 & 4 & 10.24 & 16 & 6.31 & 6 & 1.25 \\
        4 & 23,553 & 4 & 11.77 & 22 & 11.04 & 7 & 1.48 \\
        5 & 93,697 & 4 & 12.87 & 30 & 15.95 & 8 & 1.72 \\
        6 & 373,761 & 4 & 14.62 & 35 & 19.16 & 10 & 2.21 \\
        7 & 1,492,993 & 4 & 15.73 & 40 & 21.82 & 11 & 2.78 \\
        8 & 5,967,873 & 4 & 17.02 & 43 & 24.26 & 12 & 3.37 \\
        \hline
    \end{tabular}
\end{subtable}
\\[1em]
\begin{subtable}{\textwidth}
\centering
    \begin{tabular}{c c c c | c c c c}
        \hline
        \multicolumn{4}{c|}{Cell agglo}  & \multicolumn{2}{c}{$\mathcal{Q}^0$ coarse elements} & \multicolumn{2}{c}{$\mathcal{Q}^1$ coarse elements} \\
        \hline
        Ref. & DoFs & lvls & $H/h$ & Iters & Cond. & Iters & Cond. \\
        \hline
        2 & 1,521 & 4 & 10.45 & 11 & 3.73 & 6 & 1.23 \\
        3 & 5,953 & 4 & 11.13 & 17 & 6.71 & 7 & 1.43 \\
        4 & 23,553 & 4 & 12.83 & 23 & 11.74 & 8 & 1.72 \\
        5 & 93,697 & 4 & 14.01 & 30 & 16.31 & 9 & 2.04 \\
        6 & 373,761 & 4 & 15.68 & 37 & 19.64 & 10 & 2.31 \\
        7 & 1,492,993 & 4 & 16.68 & 41 & 22.19 & 11 & 2.95 \\
        8 & 5,967,873 & 4 & 17.54 & 44 & 24.79 & 12 & 3.54 \\
        \hline
    \end{tabular}
\end{subtable}
\\[1em]
\begin{subtable}{\textwidth}
\centering
    \begin{tabular}{c c c c | c c c c}
        \hline
        \multicolumn{4}{c|}{Point agglo}  & \multicolumn{2}{c}{$\mathcal{Q}^0$ coarse elements} & \multicolumn{2}{c}{$\mathcal{Q}^1$ coarse elements} \\
        \hline
        Ref. & DoFs & lvls & $H/h$ & Iters & Cond. & Iters & Cond. \\
        \hline
        2 & 1,521 & 2 & 2.92 & 11 & 3.22 & 4 & 1.05 \\
        3 & 5,953 & 2 & 3.03 & 14 & 4.07 & 5 & 1.12 \\
        4 & 23,553 & 2 & 3.51 & 17 & 5.17 & 5 & 1.25 \\
        5 & 93,697 & 2 & 4.02 & 19 & 5.76 & 6 & 1.42 \\
        6 & 373,761 & 2 & 4.19 & 21 & 6.30 & 6 & 1.57 \\
        7 & 1,492,993 & 2 & 4.62 & 23 & 7.21 & 7 & 1.76 \\
        8 & 5,967,873 & 2 & 5.00 & 26 & 8.43 & $\dagger$ & $\dagger$ \\
        \hline
    \end{tabular}
\end{subtable}
\\[1em]
\begin{subtable}{\textwidth}
\centering
    \begin{tabular}{c c c c | c c c c}
        \hline
        \multicolumn{4}{c|}{Cell agglo}  & \multicolumn{2}{c}{$\mathcal{Q}^0$ coarse elements} & \multicolumn{2}{c}{$\mathcal{Q}^1$ coarse elements} \\
        \hline
        Ref. & DoFs & lvls & $H/h$ & Iters & Cond. & Iters & Cond. \\
        \hline
        2 & 1,521 & 2 & 3.73 & 11 & 3.46 & 5 & 1.08 \\
        3 & 5,953 & 2 & 4.00 & 15 & 4.41 & 5 & 1.11 \\
        4 & 23,553 & 2 & 4.27 & 17 & 5.19 & 5 & 1.21 \\
        5 & 93,697 & 2 & 4.51 & 19 & 5.83 & 6 & 1.38 \\
        6 & 373,761 & 2 & 4.92 & 22 & 6.49 & 6 & 1.52 \\
        7 & 1,492,993 & 2 & 5.22 & 24 & 7.47 & 7 & 1.69 \\
        8 & 5,967,873 & 2 & 5.67 & 27 & 9.04 & $\dagger$ & $\dagger$ \\
        \hline
    \end{tabular}
\end{subtable}
\caption{Trying to keep $H/h$ fixed with $\mathcal{Q}^1$ elements on the unstructured square mesh. $m_{cells} = 2$ and $m_{points} = 2$. The $\dagger$ symbol indicates the solver ran out of memory.}
\label{table:fixed_unstruc_square_Q1}
\end{table}

In Table~\ref{table:ref_unstruc_square_Q1}, we show the iteration counts under $h$-refinement. Again, we notice the unbounded increase in iterations of the PCG algorithm. In contrast to the structured case, we see that neither agglomeration strategy provides a significant advantage: point-based and cell-based agglomeration produce similar results. We remark that using $p'=0$ is suboptimal. The core of our algorithm is using elements of higher order (at least one) to give a better representation of the coarse space. A high-quality coarse space allows the PCG algorithm to converge faster. 
Again, we report in Table~\ref{table:ref_unstruc_square_Q1_trilinos} the iteration counts for Trilinos AMG. The iteration counts of our multilevel preconditioner are also competitive with those of Trilinos AMG in the unstructured case, provided that a reasonable number of levels and a higher-degree coarse space are employed.

\begin{table}[!htb]
    \centering
    \begin{subtable}{\textwidth}
\centering
    \begin{tabular}{c c c c | c c c c}
        \hline
        \multicolumn{4}{c|}{Point agglo}  & \multicolumn{2}{c}{$\mathcal{Q}^0$ coarse elements} & \multicolumn{2}{c}{$\mathcal{Q}^1$ coarse elements} \\
        \hline
        Ref. & DoFs & lvls & $H/h$ & Iters & Cond. & Iters & Cond. \\
        \hline
        2 & 1,521 & 3 & 5.22 & 11 & 3.52 & 5 & 1.08 \\
        3 & 5,953 & 4 & 10.24 & 16 & 6.31 & 6 & 1.25 \\
        4 & 23,553 & 5 & 20.39 & 22 & 11.73 & 9 & 1.82 \\
        5 & 93,697 & 6 & 40.74 & 31 & 23.29 & 13 & 3.08 \\
        6 & 373,761 & 7 & 81.03 & 44 & 45.63 & 19 & 5.38 \\
        7 & 1,492,993 & 8 & 161.83 & 64 & 90.35 & 27 & 10.11 \\
        8 & 5,967,873 & 9 & 323.44 & 94 & 180.13 & 40 & 19.68 \\
        \hline
    \end{tabular}
\end{subtable}
\\[1em]
\begin{subtable}{\textwidth}
\centering
    \begin{tabular}{c c c c | c c c c}
        \hline
        \multicolumn{4}{c|}{Cell agglo}  & \multicolumn{2}{c}{$\mathcal{Q}^0$ coarse elements} & \multicolumn{2}{c}{$\mathcal{Q}^1$ coarse elements} \\
        \hline
        Ref. & DoFs & lvls & $H/h$ & Iters & Cond. & Iters & Cond. \\
        \hline
        2 & 1,521 & 3 & 5.95 & 11 & 3.73 & 5 & 1.13 \\
        3 & 5,953 & 4 & 11.13 & 17 & 6.71 & 7 & 1.43 \\
        4 & 23,553 & 5 & 21.16 & 23 & 12.33 & 9 & 1.99 \\
        5 & 93,697 & 6 & 41.40 & 31 & 23.51 & 13 & 3.17 \\
        6 & 373,761 & 7 & 81.83 & 45 & 46.07 & 19 & 5.66 \\
        7 & 1,492,993 & 8 & 162.65 & 64 & 91.18 & 28 & 10.65 \\
        8 & 5,967,873 & 9 & 324.28 & 92 & 181.67 & 40 & 20.64 \\
        \hline
    \end{tabular}
\end{subtable}
    \caption{Experimental results increasing refinements with $\mathcal{Q}^1$ elements on the unstructured square mesh. $m_{cells} = 2$ and $m_{points} = 2$.}
    \label{table:ref_unstruc_square_Q1}
\end{table}
In Table~\ref{table:fixed_unstruc_square_Q3}, we consider the two-level case and a discretization of the model problem with $\mathcal{Q}^3$ elements. Even though we are considering an unstructured mesh, we see that iteration counts remain bounded. Moreover, we see that the point-based agglomeration strategy provides an overall better performance (in terms of iteration counts). As we observed for the structured case, we see that increasing $p'$ provides a notable reduction in iteration counts. For point-based agglomeration the reduction is smaller since we also increased the size of the agglomerates.

\begin{table}[!htb]
\centering
\begin{subtable}{\textwidth}
\centering
\begin{tabular}{c c | c c c c | c c c c}
        \hline
        \multicolumn{2}{c|}{Point agglo}  & \multicolumn{4}{c|}{$\mathcal{Q}^1$ coarse elements} & \multicolumn{4}{c}{$\mathcal{Q}^2$ coarse elements} \\
        \hline
        Ref. & DoFs & lvls & $H/h$ & Iters & Cond. & lvls & $H/h$ & Iters & Cond. \\
        \hline
        2 & 13,297 & 2 & 1.60 & 8 & 1.85 & 2 & 2.33 & 6 & 1.53 \\
        3 & 52,801 & 2 & 1.69 & 8 & 2.11 & 2 & 2.97 & 7 & 2.14 \\
        4 & 210,433 & 2 & 1.78 & 8 & 2.03 & 2 & 2.58 & 6 & 1.75 \\
        5 & 840,193 & 2 & 2.11 & 9 & 2.24 & 2 & 4.62 & 9 & 4.56 \\
        6 & 3,357,697 & 2 & 1.95 & 9 & 2.44 & 2 & 3.45 & 8 & 2.63 \\
        \hline
    \end{tabular}
\end{subtable}\\[1em]
\begin{subtable}{\textwidth}
\centering
    \begin{tabular}{c c c c | c c c c}
        \hline
        \multicolumn{4}{c|}{Cell agglo}  & \multicolumn{2}{c}{$\mathcal{Q}^1$ coarse elements} & \multicolumn{2}{c}{$\mathcal{Q}^2$ coarse elements} \\
        \hline
        Ref. & DoFs & lvls & $H/h$ & Iters & Cond. & Iters & Cond. \\
        \hline
        2 & 13,297 & 2 & 3.73 & 16 & 5.19 & 9 & 2.58 \\
        3 & 52,801 & 2 & 4.00 & 18 & 5.81 & 10 & 3.09 \\
        4 & 210,433 & 2 & 4.27 & 19 & 8.38 & 10 & 3.61 \\
        5 & 840,193 & 2 & 4.51 & 22 & 10.72 & 11 & 3.99 \\
        6 & 3,357,697 & 2 & 4.92 & 24 & 13.48 & 11 & 4.41 \\
        \hline
    \end{tabular}
\end{subtable}

\caption{Experimental results trying to keep the ratio $H/h$ fixed between refinements with $\mathcal{Q}^3$ elements on the unstructured square mesh. When the coarse space has degree $p' = 1$, we set $m_{cells} = 2$ and $m_{points} = 2$. When the coarse space has degree $p' = 2$, we set $m_{cells} = 2$ and $m_{points} = 4$.}
\label{table:fixed_unstruc_square_Q3}
\end{table}
In Table~\ref{table:ref_unstruc_square_Q3}, we consider $h$-refinement with polynomial degree $p = 3$. We have that iterations increase according to the ratio $H/h$ growth. Also in this case, increasing $p'$ provides great benefits in terms of iteration counts. For $p' = 2$ we see that point-based agglomeration performs worse than cell-based: $m_{points} = 4$ allows the ratio $H/h$ to grow more than the ratio of the cell-based agglomeration. In Table~\ref{table:ref_unstruc_square_Q1_trilinos}, we show iteration counts for Trilinos AMG for $p = 3$. For the unstructured square, our preconditioner performs better than Trilinos AMG in terms of iteration counts.
\begin{table}[!htb]
\centering
\begin{subtable}{\textwidth}
\centering
\begin{tabular}{c c | c c c c | c c c c}
        \hline
        \multicolumn{2}{c|}{Point agglo}  & \multicolumn{4}{c|}{$\mathcal{Q}^1$ coarse elements} & \multicolumn{4}{c}{$\mathcal{Q}^2$ coarse elements} \\
        \hline
        Ref. & DoFs & lvls & $H/h$ & Iters & Cond. & lvls & $H/h$ & Iters & Cond. \\
        \hline
        2 & 13,297 & 3 & 3.14 & 10 & 2.25 & 3 & 6.92 & 8 & 2.07 \\
        3 & 52,801 & 4 & 6.11 & 13 & 3.34 & 4 & 18.48 & 11 & 3.23 \\
        4 & 210,433 & 5 & 12.11 & 17 & 5.53 & 5 & 42.44 & 16 & 5.74 \\
        5 & 840,193 & 6 & 24.15 & 25 & 10.10 & 5 & 52.37 & 18 & 7.25 \\
        6 & 3,357,697 & 7 & 48.18 & 34 & 18.63 & 6 & 143.44 & 29 & 12.03 \\
        \hline
    \end{tabular}
\end{subtable}\\[1em]
\begin{subtable}{\textwidth}
\centering
    \begin{tabular}{c c c c | c c c c}
        \hline
        \multicolumn{4}{c|}{Cell agglo}  & \multicolumn{2}{c}{$\mathcal{Q}^1$ coarse elements} & \multicolumn{2}{c}{$\mathcal{Q}^2$ coarse elements} \\
        \hline
        Ref. & DoFs & lvls & $H/h$ & Iters & Cond. & Iters & Cond. \\
        \hline
        2 & 13,297 & 3 & 5.95 & 17 & 5.58 & 10 & 2.63 \\
        3 & 52,801 & 4 & 11.13 & 20 & 7.17 & 11 & 2.99 \\
        4 & 210,433 & 5 & 21.16 & 25 & 9.46 & 13 & 3.37 \\
        5 & 840,193 & 6 & 41.40 & 33 & 15.82 & 15 & 4.12 \\
        6 & 3,357,697 & 7 & 81.83 & 47 & 29.74 & 20 & 6.62 \\
        \hline
    \end{tabular}
\end{subtable}

\caption{Experimental results increasing refinements with $\mathcal{Q}^3$ elements on the unstructured square mesh. When the coarse space has degree $p' = 1$, we set $m_{cells} = 2$ and $m_{points} = 2$. When the coarse space has degree $p' = 2$, we set $m_{cells} = 2$ and $m_{points} = 4$.}
\label{table:ref_unstruc_square_Q3}
\end{table}

\begin{table}[!htb]
    \centering
    \begin{subtable}{0.4\textwidth}
        \centering
    \begin{tabular}{cc}
\hline
\multicolumn{2}{c}{Trilinos AMG} \\
\hline
Ref. & Iters \\
\hline
2 & 1 \\
3 & 9 \\
4 & 12 \\
5 & 14 \\
6 & 16 \\
7 & 18 \\
8 & 20 \\
\hline
\end{tabular}
\end{subtable}
\begin{subtable}{0.4\textwidth}
\begin{tabular}{c c}
        \hline
        \multicolumn{2}{c}{Trilinos AMG} \\
        \hline
        Ref. & Iters \\
        \hline
        2 & 25 \\
        3 & 28 \\
        4 & 32 \\
        5 & 39 \\
        6 & 44 \\
        \hline
    \end{tabular}
\end{subtable}
    \caption{Trilinos AMG iterations while increasing refinements with $\mathcal{Q}^1$ elements (first table) and $\mathcal{Q}^3$ elements (second table) on the unstructured square mesh.}
    \label{table:ref_unstruc_square_Q1_trilinos}
\end{table}
    
\subsection{Unstructured square (simplices)} 
In this section, we consider an unstructured simplicial mesh of the unit square $[0,1]^2$. The  starting mesh at zero refinements is composed of 942 cells. In Table~\ref{table:fixed_unstruc_square_simplexes_Q1} we try to keep the ratio $H/h$ constant by keeping the number of levels of the multilevel preconditioner constant. As in the unstructured square case, keeping the ratio perfectly fixed is hard. We had to increase $m_{cells}$ since we needed more mesh elements for the cell-based agglomeration strategy to guarantee that our R-tree agglomeration is always admissible. Thus, for the cell-based agglomeration, we have large agglomerates and a worse $H/h$ ratio. Nevertheless, the results are similar to the previous case: we confirm that our theoretical findings also hold for non-quadrilateral meshes. 
    
\begin{table}[!htb]
\centering
\begin{subtable}{\textwidth}
\centering
    \begin{tabular}{c c c c | c c c c}
        \hline
        \multicolumn{4}{c|}{Point agglo}  & \multicolumn{2}{c}{$\mathcal{Q}^0$ coarse elements} & \multicolumn{2}{c}{$\mathcal{Q}^1$ coarse elements} \\
        \hline
        Ref. & DoFs & lvls & $H/h$ & Iters & Cond. & Iters & Cond. \\
        \hline
		2 & 7,697 & 2 & 5.16 & 16 & 5.13 & 5 & 1.13 \\
        3 & 30,465 & 2 & 5.01 & 19 & 6.09 & 5 & 1.15 \\
        4 & 121,217 & 2 & 5.24 & 21 & 6.89 & 5 & 1.16 \\
        5 & 483,585 & 2 & 5.55 & 23 & 7.47 & 5 & 1.16 \\
        6 & 1,931,777 & 2 & 5.60 & 24 & 7.84 & 5 & 1.16 \\
        \hline
    \end{tabular}
\end{subtable}
\\[1em]
\begin{subtable}{\textwidth}
\centering
    \begin{tabular}{c c c c | c c c c}
        \hline
        \multicolumn{4}{c|}{Cell agglo}  & \multicolumn{2}{c}{$\mathcal{Q}^0$ coarse elements} & \multicolumn{2}{c}{$\mathcal{Q}^1$ coarse elements} \\
        \hline
        Ref. & DoFs & lvls & $H/h$ & Iters & Cond. & Iters & Cond. \\
        \hline
		2 & 7,697 & 2 & 8.47 & 17 & 6.21 & 5 & 1.17 \\
        3 & 30,465 & 2 & 9.34 & 21 & 7.96 & 7 & 1.60 \\
        4 & 121,217 & 2 & 8.01 & 22 & 7.49 & 6 & 1.41 \\
        5 & 483,585 & 2 & 8.86 & 23 & 7.76 & 6 & 1.70 \\
        6 & 1,931,777 & 2 & 11.16 & 28 & 10.99 & 7 & 1.75 \\
        \hline
    \end{tabular}
\end{subtable}
\\[1em]

\caption{Trying to keep $H/h$ fixed with $\mathcal{P}^1$ elements on the unstructured square mesh (simplices). $m_{cells} = 4$ and $m_{points} = 2$.}
\label{table:fixed_unstruc_square_simplexes_Q1}
\end{table} 

In Table~\ref{table:ref_unstruc_square_simplexes_Q1} we show the iteration counts under $h$-refinement. We also increase the number of levels to change the ratio $H/h$. We report the iteration counts for Trilinos AMG in Table~\ref{table:ref_unstruc_square_simplexes_Q1_trilinos}. The results are similar to those of the previous unstructured square mesh case and confirm that the observations we already presented hold for a mesh of simplices. In this case too, we can see that our multilevel preconditioner performs better than Trilinos AMG as long as we use a reasonable number of levels.

\begin{table}[!htb]
\centering
\begin{subtable}{\textwidth}
\centering
    \begin{tabular}{c c c c | c c c c}
        \hline
        \multicolumn{4}{c|}{Point agglo}  & \multicolumn{2}{c}{$\mathcal{Q}^0$ coarse elements} & \multicolumn{2}{c}{$\mathcal{Q}^1$ coarse elements} \\
        \hline
        Ref. & DoFs & lvls & $H/h$ & Iters & Cond. & Iters & Cond. \\
        \hline
		2 & 7,697 & 2 & 5.16 & 16 & 5.13 & 5 & 1.13 \\
        3 & 30,465 & 3 & 9.76 & 23 & 9.75 & 7 & 1.44 \\
		4 & 121,217 & 4 & 19.39 & 32 & 18.70 & 9 & 2.00\\
		5 & 483,585 & 5 & 39.05 & 46 & 36.51 & 12 & 3.19 \\
		6 & 1,931,777 & 6 & 77.99 & 67 & 72.44 & 17 & 5.56\\
        \hline
    \end{tabular}
\end{subtable}
\\[1em]
\begin{subtable}{\textwidth}
\centering
    \begin{tabular}{c c c c | c c c c}
        \hline
        \multicolumn{4}{c|}{Cell agglo}  & \multicolumn{2}{c}{$\mathcal{Q}^0$ coarse elements} & \multicolumn{2}{c}{$\mathcal{Q}^1$ coarse elements} \\
        \hline
        Ref. & DoFs & lvls & $H/h$ & Iters & Cond. & Iters & Cond. \\
        \hline
		2 & 7,697 & 2 & 8.47 & 17 & 6.21 & 5 & 1.17 \\
        3 & 30,465 & 3 & 21.62 & 25 & 15.97 & 10 & 2.24 \\
		4 & 121,217 & 4 & 53.49 & 37 & 34.32 & 14 & 3.70 \\
		5 & 483,585 & 5 & 154.80 & 53 & 72.80 & 26 & 11.39\\
		6 & 1,931,777 & 6 & 442.33 & 75 & 144.81 & 39 & 24.46 \\
        \hline
    \end{tabular}
\end{subtable}
\\[1em]

\caption{Experimental results increasing refinements with $\mathcal{P}^1$ elements on the unstructured square mesh (simplices). $m_{cells} = 4$ and $m_{points} = 2$.}
\label{table:ref_unstruc_square_simplexes_Q1}
\end{table}  

In Table~\ref{table:fixed_unstruc_square_simplexes_Q3}, we consider the two-level preconditioner. We employ $\mathcal{P}^3$ elements and try to keep the ratio $H/h$ as stable as possible. In this case, we see that increasing $p'$ does not provide a great benefit since we also increased the size of the agglomerates. As we already observed for the unstructured square of quadrilaterals, the point-based agglomeration strategy slightly outperforms the cell-agglomeration strategy: we obtain slightly smaller agglomerates and a smaller $H/h$ ratio. However, the iteration counts are stable, as the theory suggests.

\begin{table}[!htb]
\centering
\begin{subtable}{\textwidth}
\centering
\begin{tabular}{c c | c c c c | c c c c}
        \hline
        \multicolumn{2}{c|}{Point agglo}  & \multicolumn{4}{c|}{$\mathcal{Q}^1$ coarse elements} & \multicolumn{4}{c}{$\mathcal{Q}^2$ coarse elements} \\
        \hline
        Ref. & DoFs & lvls & $H/h$ & Iters & Cond. & lvls & $H/h$ & Iters & Cond. \\
        \hline
        2 & 68,305 & 2 & 2.60 & 8 & 2.10 & 2 & 3.55 & 6 & 1.51 \\
        3 & 272,257 & 2 & 2.00 & 6 & 1.49 & 2 & 3.60 & 7 & 2.01 \\
        4 & 1,087,105 & 2 & 1.96 & 6 & 1.57 & 2 & 3.23 & 6 & 1.49 \\
        5 & 4,344,577 & 2 & 2.37 & 7 & 1.75 & 2 & 3.72 & 6 & 1.55 \\
        \hline
    \end{tabular}
\end{subtable}\\[1em]
\begin{subtable}{\textwidth}
\centering
\begin{tabular}{c c | c c c c | c c c c}
        \hline
        \multicolumn{2}{c|}{Cell agglo}  & \multicolumn{4}{c|}{$\mathcal{Q}^1$ coarse elements} & \multicolumn{4}{c}{$\mathcal{Q}^2$ coarse elements} \\
        \hline
        Ref. & DoFs & lvls & $H/h$ & Iters & Cond. & lvls & $H/h$ & Iters & Cond. \\
        \hline
        2 & 68,305 & 2 & 4.49 & 13 & 3.95 & 2 & 8.47 & 12 & 4.87 \\
        3 & 272,257 & 2 & 4.82 & 13 & 4.01 & 2 & 9.34 & 14 & 6.19 \\
        4 & 1,087,105 & 2 & 4.94 & 13 & 4.15 & 2 & 8.01 & 12 & 5.28 \\
        5 & 4,344,577 & 2 & 5.31 & 14 & 4.43 & 2 & 8.86 & 12 & 7.16 \\
        \hline
    \end{tabular}
\end{subtable}

\caption{Experimental results trying to keep the ratio $H/h$ fixed between refinements with $\mathcal{P}^3$ elements on the unstructured square mesh (simplices). When the coarse space has degree $p' = 1$, we set $m_{cells} = 2$ and $m_{points} = 2$. When the coarse space has degree $p' = 2$, we set $m_{cells} = 4$ and $m_{points} = 4$.}
\label{table:fixed_unstruc_square_simplexes_Q3}
\end{table}

In Table~\ref{table:ref_unstruc_square_simplexes_Q3}, we show the iteration counts under $h$-refinement while also increasing the number of levels (and the ratio $H/h$) for $\mathcal{P}^3$ elements. We also report in Table~\ref{table:ref_unstruc_square_simplexes_Q1_trilinos} the iteration counts for Trilinos AMG on the $\mathcal{P}^3$ discretization. The iteration counts of our preconditioner are competitive with those of Trilinos AMG as long as we choose a reasonable number of levels. In this case, increasing $p'$ does not seem to provide benefits. This is expected: we increased $m$ when we switched to $p'=2$, so the increased polynomial degree cannot fully balance the effect of the increase in agglomerate size. We emphasize this point: in our multilevel preconditioner, we must balance the size of the agglomerates with the polynomial degree of the coarse space. If we increase $p'$, we also need to provide an admissible R-tree hierarchy. It follows that we may need to increase the agglomerate size. Increasing $p'$ is not guaranteed to counterbalance the increase in $H$ needed to obtain an admissible R-tree; thus, the best preconditioner might use a lower $p'$.

\begin{table}[!htb]
\centering
\begin{subtable}{\textwidth}
\centering
\begin{tabular}{c c | c c c c | c c c c}
        \hline
        \multicolumn{2}{c|}{Point agglo}  & \multicolumn{4}{c|}{$\mathcal{Q}^1$ coarse elements} & \multicolumn{4}{c}{$\mathcal{Q}^2$ coarse elements} \\
        \hline
        Ref. & DoFs & lvls & $H/h$ & Iters & Cond. & lvls & $H/h$ & Iters & Cond. \\
        \hline
        2 & 68,305 & 2 & 2.60 & 8 & 2.10 & 2 & 3.55 & 6 & 1.51 \\
        3 & 272,257 & 3 & 3.60 & 8 & 2.01 & 3 & 9.20 & 8 & 1.94 \\
        4 & 1,087,105 & 4 & 6.94 & 10 & 2.70 & 4 & 23.05 & 10 & 2.52 \\
        5 & 4,344,577 & 5 & 15.05 & 14 & 4.35 & 5 & 90.38 & 19 & 6.67 \\
        \hline
    \end{tabular}
\end{subtable}\\[1em]
\begin{subtable}{\textwidth}
\centering
\begin{tabular}{c c | c c c c | c c c c}
        \hline
        \multicolumn{2}{c|}{Cell agglo}  & \multicolumn{4}{c|}{$\mathcal{Q}^1$ coarse elements} & \multicolumn{4}{c}{$\mathcal{Q}^2$ coarse elements} \\
        \hline
        Ref. & DoFs & lvls & $H/h$ & Iters & Cond. & lvls & $H/h$ & Iters & Cond. \\
        \hline
        2 & 68,305 & 2 & 4.49 & 13 & 3.95 & 2 & 8.47 & 12 & 4.87 \\
        3 & 272,257 & 3 & 8.08 & 14 & 4.11 & 3 & 21.62 & 15 & 6.18 \\
        4 & 1,087,105 & 4 & 15.03 & 17 & 5.30 & 4 & 53.49 & 16 & 5.06\\
        5 & 4,344,577 & 5 & 29.24 & 21 & 8.03 & 5 & 154.80 & 26 & 11.60 \\
        \hline
    \end{tabular}
\end{subtable}

\caption{Experimental results increasing refinements and levels with $\mathcal{P}^3$ elements on the unstructured square mesh (simplices). When the coarse space has degree $p' = 1$, we set $m_{cells} = 2$ and $m_{points} = 2$. When the coarse space has degree $p' = 2$, we set $m_{cells} = 4$ and $m_{points} = 4$.}
\label{table:ref_unstruc_square_simplexes_Q3}
\end{table}

\begin{table}[!htb]
    \centering
    \begin{subtable}{0.4\textwidth}
        \centering
    \begin{tabular}{cc}
\hline
\multicolumn{2}{c}{Trilinos AMG} \\
\hline
Ref. & Iters \\
\hline
2 & 9 \\
3 & 11 \\
4 & 11 \\
5 & 14 \\
6 & 15 \\
\hline
\end{tabular}
\end{subtable}
\begin{subtable}{0.4\textwidth}
\begin{tabular}{c c}
        \hline
        \multicolumn{2}{c}{Trilinos AMG} \\
        \hline
        Ref. & Iters \\
        \hline
		2 & 23 \\
        3 & 24 \\
        4 & 28 \\
        5 & 29 \\
        \hline
    \end{tabular}
\end{subtable}
    \caption{Trilinos AMG iterations while increasing refinements with $\mathcal{P}^1$ elements (first table) and $\mathcal{P}^3$ elements (second table) on the unstructured square mesh (simplices).}
    \label{table:ref_unstruc_square_simplexes_Q1_trilinos}
\end{table}

\subsection{Structured cube}
Here, we consider the case of a structured discretization of the unit cube $[0,1]^3$. In this instance, the starting mesh at zero refinements is the single cell unit cube. In Table~\ref{table:fixed_struc_cube_Q1}, we try to keep the ratio $H/h$ constant. In Table~\ref{table:ref_struc_cube_Q1}, instead, we show the iteration counts under $h$-refinement. In both cases, our algorithm behaves similarly to the structured square case. In Table~\ref{table:ref_struc_cube_Q1_trilinos}, we report the iteration counts for Trilinos AMG. Again, our algorithm is competitive with Trilinos AMG in terms of iteration counts.

\begin{table}[!htb]
\centering
\begin{subtable}{\textwidth}
\centering
    \begin{tabular}{c c c c | c c c c}
        \hline
        \multicolumn{4}{c|}{Point agglo}  & \multicolumn{2}{c}{$\mathcal{Q}^0$ coarse elements} & \multicolumn{2}{c}{$\mathcal{Q}^1$ coarse elements} \\
        \hline
        Ref. & DoFs & lvls & $H/h$ & Iters & Cond. & Iters & Cond. \\
        \hline
        4 & 4,913 & 2 & 4.83 & 4 & 1.27 & 4 & 1.04 \\
        5 & 35,937 & 2 & 5.10 & 7 & 2.35 & 4 & 1.05 \\
        6 & 274,625 & 2 & 4.97 & 11 & 3.52 & 4 & 1.05 \\
        7 & 2,146,689 & 2 & 6.06 & 14 & 4.18 & $\dagger$ & $\dagger$ \\
        \hline
    \end{tabular}
\end{subtable}
\\[1em]
\begin{subtable}{\textwidth}
\centering
    \begin{tabular}{c c c c | c c c c}
        \hline
        \multicolumn{4}{c|}{Cell agglo}  & \multicolumn{2}{c}{$\mathcal{Q}^0$ coarse elements} & \multicolumn{2}{c}{$\mathcal{Q}^1$ coarse elements} \\
        \hline
        Ref. & DoFs & lvls & $H/h$ & Iters & Cond. & Iters & Cond. \\
        \hline
        4 & 4,913 & 2 & 4.00 & 4 & 1.28 & 4 & 1.04 \\
        5 & 35,937 & 2 & 4.00 & 7 & 2.33 & 4 & 1.03 \\
        6 & 274,625 & 2 & 4.00 & 10 & 3.41 & 4 & 1.03 \\
        7 & 2,146,689 & 2 & 4.00 & 13 & 3.90 & $\dagger$ & $\dagger$ \\
        \hline
    \end{tabular}
\end{subtable}
\caption{Trying to keep $H/h$ fixed with $\mathcal{Q}^1$ elements on the structured cube mesh. $m_{cells} = 4$ and $m_{points} = 4$. The $\dagger$ symbol indicates the solver ran out of memory.}
\label{table:fixed_struc_cube_Q1}
\end{table}

\begin{table}[!htb]
\centering
\begin{subtable}{\textwidth}
\centering
    \begin{tabular}{c c c c | c c c c}
        \hline
        \multicolumn{4}{c|}{Point agglo}  & \multicolumn{2}{c}{$\mathcal{Q}^0$ coarse elements} & \multicolumn{2}{c}{$\mathcal{Q}^1$ coarse elements} \\
        \hline
        Ref. & DoFs & lvls & $H/h$ & Iters & Cond. & Iters & Cond. \\
        \hline
        4 & 4,913 & 3 & 8.70 & 4 & 1.27 & 4 & 1.04 \\
        5 & 35,937 & 4 & 17.34 & 7 & 2.38 & 5 & 1.14 \\
        6 & 274,625 & 5 & 31.85 & 11 & 4.52 & 7 & 1.40 \\
        7 & 2,146,689 & 6 & 63.76 & 17 & 8.88 & 10 & 2.18 \\
        \hline
    \end{tabular}
\end{subtable}
\\[1em]
\begin{subtable}{\textwidth}
\centering
    \begin{tabular}{c c c c | c c c c}
        \hline
        \multicolumn{4}{c|}{Cell agglo}  & \multicolumn{2}{c}{$\mathcal{Q}^0$ coarse elements} & \multicolumn{2}{c}{$\mathcal{Q}^1$ coarse elements} \\
        \hline
        Ref. & DoFs & lvls & $H/h$ & Iters & Cond. & Iters & Cond. \\
        \hline
        4 & 4,913 & 2 & 4.00 & 4 & 1.28 & 4 & 1.04 \\
        5 & 35,937 & 3 & 8.00 & 7 & 2.36 & 4 & 1.03 \\
        6 & 274,625 & 4 & 16.00 & 10 & 4.40 & 5 & 1.13 \\
        7 & 2,146,689 & 5 & 32.00 & 15 & 8.46 & 7 & 1.53 \\
        \hline
    \end{tabular}
\end{subtable}
\caption{Experimental results increasing refinements with $\mathcal{Q}^1$ elements on the structured cube mesh. $m_{cells} = 4$ and $m_{points} = 4$.}
\label{table:ref_struc_cube_Q1}
\end{table}

In Table~\ref{table:fixed_struc_cube_Q2} and~\ref{table:ref_struc_cube_Q2}, we use $\mathcal{Q}^2$ conforming elements to discretize the model problem. We use $p' = 1,2$ as coarse space degree. Again, when $p'=2$, we increase $m_{points}$. In the three dimensional case we could not test the two-level preconditioner since we would easily go out of memory on our machine. Nonetheless, we see that the iteration count is stable in all cases. Observations similar to those for the structured square mesh could be made. In Table~\ref{table:ref_struc_cube_Q1_trilinos} we show the iteration counts of Trilinos AMG: clearly our preconditioner outperforms Trilinos AMG in terms of iteration counts.

\begin{table}[!htb]
\centering
\begin{subtable}{\textwidth}
\centering
    \begin{tabular}{c c | c c c c | c c c c}
        \hline
        \multicolumn{2}{c|}{Point agglo}  & \multicolumn{4}{c|}{$\mathcal{Q}^1$ coarse elements} & \multicolumn{4}{c}{$\mathcal{Q}^2$ coarse elements} \\
        \hline
        Ref. & DoFs & lvls & $H/h$ & Iters & Cond. & lvls & $H/h$ & Iters & Cond. \\
        \hline
        4 & 35,937 & 3 & 5.10 & 6 & 1.25 & 3 & 8.19 & 5 & 1.21 \\
        5 & 274,625 & 3 & 4.73 & 5 & 1.13 & 3 & 9.09 & 6 & 1.36 \\
        6 & 2,146,689 & 3 & 4.73 & 6 & 1.43 & 3 & 10.34 & 6 & 1.52 \\
        \hline
    \end{tabular}
\end{subtable}\\[1em]
\begin{subtable}{\textwidth}
\centering
    \begin{tabular}{c c c c | c c c c}
        \hline
        \multicolumn{4}{c|}{Cell agglo}  & \multicolumn{2}{c}{$\mathcal{Q}^1$ coarse elements} & \multicolumn{2}{c}{$\mathcal{Q}^2$ coarse elements} \\
        \hline
        Ref. & DoFs & lvls & $H/h$ & Iters & Cond. & Iters & Cond. \\
        \hline
        4 & 35,937 & 3 & 8 & 7 & 1.39 & 4 & 1.04 \\
        5 & 274,625 & 3 & 8 & 7 & 1.55 & 4 & 1.05 \\
        6 & 2,146,689 & 3 & 8 & 7 & 1.60 & 4 & 1.05 \\
        \hline
    \end{tabular}
\end{subtable}

\caption{Experimental results trying to keep the ratio $H/h$ fixed between refinements with $\mathcal{Q}^2$ elements on the structured cube mesh. When the coarse space has degree $p' = 1$ we set $m_{cells} = 4$ and $m_{points} = 4$. When the coarse space has degree $p' = 2$ we set $m_{cells} = 4$ and $m_{points} = 7$.}
\label{table:fixed_struc_cube_Q2}
\end{table}

\begin{table}[!htb]
\centering
\begin{subtable}{\textwidth}
\centering
    \begin{tabular}{c c | c c c c | c c c c}
        \hline
        \multicolumn{2}{c|}{Point agglo}  & \multicolumn{4}{c|}{$\mathcal{Q}^1$ coarse elements} & \multicolumn{4}{c}{$\mathcal{Q}^2$ coarse elements} \\
        \hline
        Ref. & DoFs & lvls & $H/h$ & Iters & Cond. & lvls & $H/h$ & Iters & Cond. \\
        \hline
        4 & 35,937 & 4 & 8.67 & 6 & 1.32 & 2 & 4.09 & 4 & 1.10 \\
        5 & 274,625 & 5 & 15.92 & 8 & 1.70 & 3 & 9.09 & 6 & 1.36 \\
        6 & 2,146,689 & 6 & 31.83 & 12 & 2.73 & 4 & 21.40 & 8 & 1.78 \\
        \hline
    \end{tabular}
\end{subtable}\\[1em]
\begin{subtable}{\textwidth}
\centering
    \begin{tabular}{c c c c | c c c c}
        \hline
        \multicolumn{4}{c|}{Cell agglo}  & \multicolumn{2}{c}{$\mathcal{Q}^1$ coarse elements} & \multicolumn{2}{c}{$\mathcal{Q}^2$ coarse elements} \\
        \hline
        Ref. & DoFs & lvls & $H/h$ & Iters & Cond. & Iters & Cond. \\
        \hline
        4 & 35,937 & 2 & 4 & 5 & 1.33 & 4 & 1.04 \\
        5 & 274,625 & 3 & 8 & 7 & 1.55 & 4 & 1.05 \\
        6 & 2,146,689 & 4 & 16 & 9 & 2.08 & 5 & 1.07 \\
        \hline
    \end{tabular}
\end{subtable}

\caption{Experimental results increasing refinements with $\mathcal{Q}^2$ elements on the structured cube mesh. When the coarse space has degree $p' = 1$ we set $m_{cells} = 4$ and $m_{points} = 4$. When the coarse space has degree $p' = 2$ we set $m_{cells} = 4$ and $m_{points} = 7$.}
\label{table:ref_struc_cube_Q2}
\end{table}

\begin{table}[!htb]
    \centering
    \begin{subtable}{0.4\textwidth}
    \centering
    \begin{tabular}{cc}
\hline
\multicolumn{2}{c}{Trilinos AMG} \\
\hline
Ref. & Iters \\
\hline
4 & 12 \\
5 & 10 \\
6 & 12 \\
7 & 12 \\
\hline
\end{tabular}
    \end{subtable}
    \begin{subtable}{0.4\textwidth}
    \centering
\begin{tabular}{c c}
        \hline
        \multicolumn{2}{c}{Trilinos AMG} \\
        \hline
        Ref. & Iters \\
        \hline
        4 & 18 \\
        5 & 18 \\
        6 & 19 \\
        \hline
    \end{tabular}
    \end{subtable}
    \caption{Trilinos AMG iterations while increasing refinements with $\mathcal{Q}^1$ elements (first table) and $\mathcal{Q}^2$ elements (second table) on the structured cube mesh.}
    \label{table:ref_struc_cube_Q1_trilinos}
\end{table}

\subsection{Idealized left ventricle mesh}
To show the applicability of our algorithm to more general or realistic cases, we test our multilevel preconditioner on an idealized mesh of the heart's left ventricle. The mesh discretizes the space between two nested ellipsoids cut by a plane. The starting mesh at zero refinements is composed of 50,988 hexahedral cells. There is little difference between Table~\ref{table:fixed_id_lv_mesh_Q1} and Table~\ref{table:ref_id_lv_mesh_Q1} in terms of iteration counts or conditioning. This means that the increase in agglomerate size is not large enough to reduce the impact of the coarse-space correction. Despite this, our preconditioner is effective, and its iteration counts are competitive with those of Trilinos AMG in Table~\ref{table:ref_id_lv_mesh_Q1_trilinos}.

\begin{table}[!htb]
    \centering
    \begin{subtable}{\textwidth}
\centering
    \begin{tabular}{c c c c | c c c c}
        \hline
        \multicolumn{4}{c|}{Point agglo}  & \multicolumn{2}{c}{$\mathcal{Q}^0$ coarse elements} & \multicolumn{2}{c}{$\mathcal{Q}^1$ coarse elements} \\
        \hline
        Ref. & DoFs & lvls & $H/h$ & Iters & Cond. & Iters & Cond. \\
        \hline
        0 & 58,776 & 3 & 10.65 & 5 & 1.24 & 4 & 1.08 \\
        1 & 438,915 & 3 & 13.81 & 9 & 2.43 & 6 & 1.36 \\
        2 & 3,386,997 & 3 & 22.82 & 13 & 4.52 & 8 & 2.09 \\
        \hline
    \end{tabular}
\end{subtable}
\\[1em]
\begin{subtable}{\textwidth}
\centering
    \begin{tabular}{c c c c | c c c c}
        \hline
        \multicolumn{4}{c|}{Cell agglo}  & \multicolumn{2}{c}{$\mathcal{Q}^0$ coarse elements} & \multicolumn{2}{c}{$\mathcal{Q}^1$ coarse elements} \\
        \hline
        Ref. & DoFs & lvls & $H/h$ & Iters & Cond. & Iters & Cond. \\
        \hline
        0 & 58,776 & 3 & 11.77 & 5 & 1.24 & 4 & 1.10 \\
        1 & 438,915 & 3 & 14.62 & 9 & 2.52 & 6 & 1.44 \\
        2 & 3,386,997 & 3 & 17.90 & 14 & 4.79 & 7 & 1.73 \\
        \hline
    \end{tabular}
\end{subtable}
    \caption{Trying to keep $H/h$ fixed with $\mathcal{Q}^1$ elements on the idealized left ventricle mesh. $m_{cells} = 4$ and $m_{points} = 4$.}
    \label{table:fixed_id_lv_mesh_Q1}
\end{table}

\begin{table}[!htb]
    \centering
    \begin{subtable}{\textwidth}
\centering
    \begin{tabular}{c c c c | c c c c}
        \hline
        \multicolumn{4}{c|}{Point agglo}  & \multicolumn{2}{c}{$\mathcal{Q}^0$ coarse elements} & \multicolumn{2}{c}{$\mathcal{Q}^1$ coarse elements} \\
        \hline
        Ref. & DoFs & lvls & $H/h$ & Iters & Cond. & Iters & Cond. \\
        \hline
        0 & 58,776 & 3 & 10.65 & 5 & 1.24 & 4 & 1.08 \\
        1 & 438,915 & 4 & 22.63 & 9 & 2.43 & 6 & 1.38 \\
        2 & 3,386,997 & 5 & 46.72 & 13 & 4.53 & 9 & 2.15 \\
        \hline
    \end{tabular}
\end{subtable}
\\[1em]
\begin{subtable}{\textwidth}
\centering
    \begin{tabular}{c c c c | c c c c}
        \hline
        \multicolumn{4}{c|}{Cell agglo}  & \multicolumn{2}{c}{$\mathcal{Q}^0$ coarse elements} & \multicolumn{2}{c}{$\mathcal{Q}^1$ coarse elements} \\
        \hline
        Ref. & DoFs & lvls & $H/h$ & Iters & Cond. & Iters & Cond. \\
        \hline
        0 & 58,776 & 3 & 11.77 & 5 & 1.24 & 4 & 1.10 \\
        1 & 438,915 & 4 & 23.48 & 9 & 2.52 & 6 & 1.49 \\
        2 & 3,386,997 & 5 & 46.24 & 14 & 4.79 & 9 & 2.15 \\
        \hline
    \end{tabular}
\end{subtable}
    \caption{Experimental results increasing refinements with $\mathcal{Q}^1$ elements on the idealized left ventricle mesh. $m_{cells} = 4$ and $m_{points} = 4$.}
    \label{table:ref_id_lv_mesh_Q1}
\end{table}

In Table~\ref{table:fixed_id_lv_mesh_Q2} and Table~\ref{table:ref_id_lv_mesh_Q2}, we use $\mathcal{Q}^2$ conforming finite elements to discretize the model problem. We show different combinations of number of levels: it is not possible to refine the mesh more than once without running out of memory on our machine. In any case, our iteration counts outperform Trilinos AMG iterations reported in Table~\ref{table:ref_id_lv_mesh_Q1_trilinos}.

\begin{table}[!htb]
\centering
\begin{subtable}{\textwidth}
\centering
\begin{tabular}{c c | c c c c | c c c c}
        \hline
        \multicolumn{2}{c|}{Point agglo}  & \multicolumn{4}{c|}{$\mathcal{Q}^1$ coarse elements} & \multicolumn{4}{c}{$\mathcal{Q}^2$ coarse elements} \\
        \hline
        Ref. & DoFs & lvls & $H/h$ & Iters & Cond. & lvls & $H/h$ & Iters & Cond. \\
        \hline
        0 & 438,915 & 3 & 7.04 & 7 & 1.67 & 3 & 9.69 & 7 & 1.70  \\
        1 & 3,386,997 & 3 & 11.53 & 9 & 2.70 & 3 & 20.63 & 9 & 3.00  \\
        \hline
    \end{tabular}
\end{subtable}\\[1em]
\begin{subtable}{\textwidth}
\centering
\begin{tabular}{c c c c | c c c c}
        \hline
        \multicolumn{4}{c|}{Cell agglo}  & \multicolumn{2}{c}{$\mathcal{Q}^1$ coarse elements} & \multicolumn{2}{c}{$\mathcal{Q}^2$ coarse elements} \\
        \hline
        Ref. & DoFs & lvls & $H/h$ & Iters & Cond. & Iters & Cond. \\
        \hline
        0 & 438,915 & 3 & 11.77 & 10 & 2.53 & 7 & 1.81 \\
        1 & 3,386,997 & 3 & 14.62 & 12 & 3.83 &  9 & 2.45 \\
        \hline
    \end{tabular}
\end{subtable}

\caption{Experimental results trying to keep the ratio $H/h$ fixed between refinements with $\mathcal{Q}^2$ elements on the idealized left ventricle mesh. When the coarse space has degree $p' = 1$ we set $m_{cells} = 4$ and $m_{points} = 4$. When the coarse space has degree $p' = 2$ we set $m_{cells} = 4$ and $m_{points} = 7$.}
\label{table:fixed_id_lv_mesh_Q2}
\end{table}

\begin{table}[!htb]
\centering
\begin{subtable}{\textwidth}
\centering
\begin{tabular}{c c | c c c c | c c c c}
        \hline
        \multicolumn{2}{c|}{Point agglo}  & \multicolumn{4}{c|}{$\mathcal{Q}^1$ coarse elements} & \multicolumn{4}{c}{$\mathcal{Q}^2$ coarse elements} \\
        \hline
        Ref. & DoFs & lvls & $H/h$ & Iters & Cond. & lvls & $H/h$ & Iters & Cond. \\
        \hline
        0 & 438,915 & 5 & 20.92 & 7 & 1.72 & 3 & 9.69 & 7 & 1.70 \\
        1 & 3,386,997 & 6 & 41.03 & 10 & 2.88 & 4 & 28.90 & 10 & 3.11 \\
        \hline
    \end{tabular}
\end{subtable}\\[1em]
\begin{subtable}{\textwidth}
\centering
    \begin{tabular}{c c c c | c c c c}
        \hline
        \multicolumn{4}{c|}{Cell agglo}  & \multicolumn{2}{c}{$\mathcal{Q}^1$ coarse elements} & \multicolumn{2}{c}{$\mathcal{Q}^2$ coarse elements} \\
        \hline
        Ref. & DoFs & lvls & $H/h$ & Iters & Cond. & Iters & Cond. \\
        \hline
        0 & 438,915 & 3 & 11.77 & 10 & 2.53 & 7 & 1.81 \\
        1 & 3,386,997 & 4 & 23.48 & 12 & 3.99 & 9 & 2.75 \\
        \hline
    \end{tabular}
\end{subtable}
\caption{Experimental results increasing refinements with $\mathcal{Q}^2$ elements on the idealized left ventricle mesh. When the coarse space has degree $p' = 1$ we set $m_{cells} = 4$ and $m_{points} = 4$. When the coarse space has degree $p' = 2$ we set $m_{cells} = 4$ and $m_{points} = 7$.}
\label{table:ref_id_lv_mesh_Q2}
\end{table}

\begin{table}[!htb]
    \centering
    \begin{subtable}{0.4\textwidth}
    \centering
        \begin{tabular}{cc}
        \hline
        \multicolumn{2}{c}{Trilinos AMG} \\
        \hline
        Ref. & Iters \\
        \hline
        0 & 12 \\
        1 & 12 \\
        2 & 13 \\
        \hline
    \end{tabular}
 \end{subtable}
 \begin{subtable}{0.4\textwidth}
 \centering
 \begin{tabular}{c c}
        \hline
        \multicolumn{2}{c}{Trilinos AMG} \\
        \hline
        Ref. & Iters \\
        \hline
        0 & 15 \\
        1 & 15 \\
        \hline
    \end{tabular}
 \end{subtable}
    \caption{Trilinos AMG iterations while increasing refinements with $\mathcal{Q}^1$ elements (first table) and $\mathcal{Q}^2$ elements (second table) on the idealized left ventricle mesh.}
    \label{table:ref_id_lv_mesh_Q1_trilinos}
\end{table}

\subsection{Realistic left ventricle mesh}
In this part, we consider a realistic mesh of the left ventricle of the heart to test the multilevel preconditioner on a mesh of practical interest. We stress that, in a realistic pipeline, the starting mesh is already given and fine enough. The mesh at zero refinements is composed of 374,022 hexahedral cells. From Table~\ref{table:fixed_real_lv_mesh_Q1}, we see that the performance is similar for both agglomeration strategies. We remark that, due to the mesh properties, it was not possible to refine the mesh more than once and run our algorithm with the sequential implementation we have in place. From Table~\ref{table:ref_real_lv_mesh_Q1_trilinos} we can see that even for a realistic mesh our algorithm is competitive with Trilinos AMG in terms of iteration counts.

\begin{table}[!htb]
    \centering
    \begin{subtable}{\textwidth}
\centering
    \begin{tabular}{c c c c | c c c c}
        \hline
        \multicolumn{4}{c|}{Point agglo}  & \multicolumn{2}{c}{$\mathcal{Q}^0$ coarse elements} & \multicolumn{2}{c}{$\mathcal{Q}^1$ coarse elements} \\
        \hline
        Ref. & DoFs & lvls & $H/h$ & Iters & Cond. & Iters & Cond. \\
        \hline
        0 & 409,752 & 4 & 19.46 & 10 & 2.85 & 7 & 1.62 \\
        1 & 3,134,785 & 4 & 23.74 & 14 & 5.19 & 9 & 2.34 \\
        1 & 3,134,785 & 5 & 35.05 & 14 & 5.19 & 9 & 2.48 \\
        \hline
    \end{tabular}
\end{subtable}
\\[1em]
\begin{subtable}{\textwidth}
\centering
    \begin{tabular}{c c c c | c c c c}
        \hline
        \multicolumn{4}{c|}{Cell agglo}  & \multicolumn{2}{c}{$\mathcal{Q}^0$ coarse elements} & \multicolumn{2}{c}{$\mathcal{Q}^1$ coarse elements} \\
        \hline
        Ref. & DoFs & lvls & $H/h$ & Iters & Cond. & Iters & Cond. \\
        \hline
        0 & 409,752 & 4 & 20.78 & 11 & 3.02 & 7 & 1.72 \\
        1 & 3,134,785 & 4 & 23.89 & 15 & 5.49 & 9 & 2.43 \\
        1 & 3,134,785 & 5 & 36.09 & 15 & 5.49 & 10 & 2.65 \\
        \hline
    \end{tabular}
\end{subtable}
    \caption{Experimental results with $\mathcal{Q}^1$ elements on the realistic left ventricle mesh. $m_{cells} = 4$ and $m_{points} = 4$.}
    \label{table:fixed_real_lv_mesh_Q1}
\end{table}

In Table~\ref{table:fixed_real_lv_mesh_Q2}, we use $\mathcal{Q}^2$ conforming elements to discretize the model problem. Since we increase $p$, we cannot do any $h$-refinement on this mesh or we would run out of memory on our machine. Regardless, we show that our preconditioner works on a non-standard mesh with finite element order greater than one. From Table~\ref{table:ref_real_lv_mesh_Q1_trilinos} we can also see that the iteration counts of PCG preconditioned with our multilevel preconditioner outperform PCG preconditioned with Trilinos AMG.

\begin{table}[!htb]
\centering
\begin{subtable}{\textwidth}
\centering
    \begin{tabular}{c c | c c c c | c c c c}
        \hline
        \multicolumn{2}{c|}{Point agglo}  & \multicolumn{4}{c|}{$\mathcal{Q}^1$ coarse elements} & \multicolumn{4}{c}{$\mathcal{Q}^2$ coarse elements} \\
        \hline
        Ref. & DoFs & lvls & $H/h$ & Iters & Cond. & lvls & $H/h$ & Iters & Cond. \\
        \hline
        0 & 3,134,785 & 4 & 13.52 & 10 & 3.07 & 3 & 24.41 & 10 & 3.44\\
        \hline
    \end{tabular}
\end{subtable}
\\[1em]
\begin{subtable}{\textwidth}
\centering
    \begin{tabular}{c c | c c c c | c c c c}
        \hline
        \multicolumn{2}{c|}{Cell agglo}  & \multicolumn{4}{c|}{$\mathcal{Q}^1$ coarse elements} & \multicolumn{4}{c}{$\mathcal{Q}^2$ coarse elements} \\
        \hline
        Ref. & DoFs & lvls & $H/h$ & Iters & Cond. & lvls & $H/h$ & Iters & Cond. \\
        \hline
        0 & 3,134,785 & 3 & 13.68 & 13 & 4.92 & 4 & 20.78 & 11 & 3.83 \\
        \hline
    \end{tabular}
\end{subtable}
\caption{Experimental results with $\mathcal{Q}^2$ elements on the realistic left ventricle mesh. When the coarse space has degree $p' = 1$ we set $m_{cells} = 4$ and $m_{points} = 4$. When the coarse space has degree $p' = 2$ we set $m_{cells} = 4$ and $m_{points} = 7$.}
\label{table:fixed_real_lv_mesh_Q2}
\end{table}

\begin{table}[!htb]
    \centering
    \begin{subtable}{0.4\textwidth}
    \centering
    \begin{tabular}{cc}
        \hline
        \multicolumn{2}{c}{Trilinos AMG} \\
        \hline
        Ref. & Iters \\
        \hline
        0 & 10 \\
        1 & 13 \\
        \hline
    \end{tabular}
    \end{subtable}
    \begin{subtable}{0.4\textwidth}
    \centering
        \begin{tabular}{cc}
        \hline
        \multicolumn{2}{c}{Trilinos AMG} \\
        \hline
        Ref. & Iters \\
        \hline
        0 & 16 \\
        \hline
    \end{tabular}
    \end{subtable}
    \caption{Trilinos AMG iterations while increasing refinements with $\mathcal{Q}^1$ elements (first table) and $\mathcal{Q}^2$ elements (second table) on the realistic left ventricle mesh.}
    \label{table:ref_real_lv_mesh_Q1_trilinos}
\end{table}

\subsection{Liver mesh (simplices)}
Finally, we consider a realistic simplicial mesh representing a human liver. We want to verify that our preconditioner works on a nontrivial simplicial mesh of practical interest. At zero refinements, the mesh is composed of 284,201 simplicial cells. In Table~\ref{table:fixed_liver_mesh_Q1}, we report results with $p=1$. Even though we keep the number of levels fixed, the ratio $H/h$ and the iteration counts increase. Nevertheless, our iteration counts are comparable with those of Trilinos AMG in Table~\ref{table:ref_liver_mesh_Q1_trilinos}.

\begin{table}[!htb]
    \centering
    \begin{subtable}{\textwidth}
\centering
    \begin{tabular}{c c c c | c c c c}
        \hline
        \multicolumn{4}{c|}{Point agglo}  & \multicolumn{2}{c}{$\mathcal{Q}^0$ coarse elements} & \multicolumn{2}{c}{$\mathcal{Q}^1$ coarse elements} \\
        \hline
        Ref. & DoFs & lvls & $H/h$ & Iters & Cond. & Iters & Cond. \\
        \hline
       	0 & 51,339 & 3 & 17.06 & 13 & 4.74 & 7 & 1.60 \\
        1 & 397,238 & 3 & 20.04 & 18 & 8.41 & 9 & 2.45 \\
        2 & 3,109,535 & 3 & 30.58 & 22 & 11.60 & 11 & 3.69 \\
        \hline
    \end{tabular}
\end{subtable}
\\[1em]
\begin{subtable}{\textwidth}
\centering
    \begin{tabular}{c c c c | c c c c}
        \hline
        \multicolumn{4}{c|}{Cell agglo}  & \multicolumn{2}{c}{$\mathcal{Q}^0$ coarse elements} & \multicolumn{2}{c}{$\mathcal{Q}^1$ coarse elements} \\
        \hline
        Ref. & DoFs & lvls & $H/h$ & Iters & Cond. & Iters & Cond. \\
        \hline
        0 & 51,339 & 3 & 14.31 & 11 & 3.78 & 6 & 1.41\\
        1 & 397,238 & 3 & 21.71 & 14 & 5.07 & 7 & 1.88 \\
        2 & 3,109,535 & 3 & 38.35 & 18 & 7.70 & $\dagger$ & $\dagger$  \\
        \hline
    \end{tabular}
\end{subtable}
    \caption{Experimental results with $\mathcal{P}^1$ elements on the liver mesh (simplices). $m_{cells} = 4$ and $m_{points} = 4$. The $\dagger$ symbol means that the solver ran out of memory.}
    \label{table:fixed_liver_mesh_Q1}
\end{table}
    
    In Table~\ref{table:fixed_liver_mesh_Q2} we consider the case $p=2$. Due to memory issues, we could not refine the mesh more than once. Nonetheless, we can see that our iteration counts are competitive with the Trilinos AMG implementation. When $p'=2$, we do not get an improvement in terms of iteration counts because we also increase $m_{cells}$ and $m_{points}$ and obtain a larger $H/h$ ratio. The increase in the ratio is not fully compensated by increasing $p'$. 
    
    \begin{table}[!htb]
    \centering
    \begin{subtable}{\textwidth}
\centering        
        \begin{tabular}{c c | c c c c | c c c c}
        \hline
        \multicolumn{2}{c|}{Point agglo}  & \multicolumn{4}{c|}{$\mathcal{Q}^1$ coarse elements} & \multicolumn{4}{c}{$\mathcal{Q}^2$ coarse elements} \\
        \hline
        Ref. & DoFs & lvls & $H/h$ & Iters & Cond. & lvls & $H/h$ & Iters & Cond. \\
        \hline
        0 & 397,238 & 3 & 10.70 & 10 & 2.91 & 3 & 17.11 & 10 & 3.10 \\
        1 & 3,109,535 & 3 & 15.29 & 12 & 4.44 & 3 & 23.98 & 12 & 4.56 \\
        \hline
\end{tabular}
 \end{subtable}
\\[1em]
\begin{subtable}{\textwidth}      
\centering 
\begin{tabular}{c c | c c c c | c c c c}
        \hline
        \multicolumn{2}{c|}{Cell agglo}  & \multicolumn{4}{c|}{$\mathcal{Q}^1$ coarse elements} & \multicolumn{4}{c}{$\mathcal{Q}^2$ coarse elements} \\
        \hline
        Ref. & DoFs & lvls & $H/h$ & Iters & Cond. & lvls & $H/h$ & Iters & Cond. \\
        \hline
        0 & 397,238 & 3 & 14.31 & 11 & 3.54 & 3 & 18.93 & 10 & 3.21\\
        1 & 3,109,535 & 3 & 21.71 & 14 & 5.35 & 3 & 28.72 & 12 & 4.29 \\
        \hline
    \end{tabular}

\end{subtable}
    \caption{Experimental results with $\mathcal{P}^2$ elements on the liver mesh (simplices). When $p' = 1$ we set $m_{cells} = 4$ and $m_{points} = 4$. When $p' = 2$ we set $m_{cells} = 6$ and $m_{points} = 6$.}
    \label{table:fixed_liver_mesh_Q2}
\end{table}

\begin{table}[!htb]
    \centering
    \begin{subtable}{0.4\textwidth}
    \centering
    \begin{tabular}{cc}
        \hline
        \multicolumn{2}{c}{Trilinos AMG} \\
        \hline
        Ref. & Iters \\
        \hline
		0 & 9 \\
        1 & 10 \\
        2 & 10 \\
        \hline
    \end{tabular}
    \end{subtable}
    \begin{subtable}{0.4\textwidth}
    \centering
        \begin{tabular}{cc}
        \hline
        \multicolumn{2}{c}{Trilinos AMG} \\
        \hline
        Ref. & Iters \\
        \hline
		0 & 13 \\
        1 & 14 \\
        \hline
    \end{tabular}
    \end{subtable}
    \caption{Trilinos AMG iterations while increasing refinements with $\mathcal{P}^1$ elements (first table) and $\mathcal{P}^2$ elements (second table) on the liver mesh (simplices).}
    \label{table:ref_liver_mesh_Q1_trilinos}
\end{table}

\section{Conclusions and outlook}\label{sec:conc}
\noindent 
In this paper, we introduced a novel multilevel preconditioner for high-order conforming finite element discretizations. While the preconditioner at its core is algebraic, and no prescribed hierarchy is needed, we also exploit geometrical information from the mesh in order to build enriched coarse spaces. The augmented coarse space retains the polynomial information needed to deal with high-order discretizations. We also presented the convergence analysis of the two-level preconditioner. The analysis shows the benefits of increasing the polynomial degree of the coarse space. Even though the estimates we proved are not sharp, we showed in the numerical experiments that our preconditioner performs better than theory suggests and is competitive, in terms of iteration counts, with the Trilinos implementation of AMG. The loose bounds we obtained guarantee uniform convergence if the ratio $H/h$ and polynomial degree $p$ are fixed but are not sharp enough to give clear indications on the performance of the preconditioner.
Even though some parameter tuning is needed for our multilevel preconditioner (the choice of $m$ in the R-tree agglomeration is crucial for the behavior of the preconditioner), we proved that our preconditioner performs effectively on high-order discretizations. We believe that the multilevel preconditioner we presented in this work represents a viable alternative for preconditioning the linear systems arising from high-order finite element discretizations. 

\section*{Data availability}
\noindent The source code implementing the methodology and the experiments showcased in this paper are publicly available at the GitHub page of \textsc{polyDEAL}~\cite{polydeal}. Detailed instructions for compilation and execution are provided in the repository.

\section*{Acknowledgments}
\noindent MF and LH acknowledge support of the European Research Council (ERC) under the European Union's Horizon 2020 research and innovation programme (call HORIZON-EUROHPC-JU-2023-COE-03, grant agreement No. 101172493 ``dealii-X'').  LH acknowledges partial support by King Abdullah University of Science and Technology Research Funding (KRF) under Award No. ORFS-2025-CRG13-6911.3. The authors are members of Gruppo Nazionale per il Calcolo Scientifico (GNCS), Istituto Nazionale di Alta Matematica (INdAM).

\section*{Use of AI tools in the writing process}
\noindent During the preparation of this work, the authors used ChatGPT v5.6 in order to improve readability and language. The authors reviewed and edited the content as needed and take full responsibility for the content of the published article.

\bibliographystyle{abbrv}
\bibliography{refs}
\end{document}